\documentclass[11pt,a4paper]{amsart}
\usepackage[DIV14,paper=a4,BCOR14mm,headinclude]{typearea}
\usepackage[british]{babel}
\usepackage{amssymb}
\usepackage{graphicx}
\usepackage{epstopdf}
\usepackage{color}
\usepackage{tikz}
\usetikzlibrary{patterns}
\usepackage{diagbox}
\usepackage{pgfplots}
\usepackage{pgfplotstable}
\usepackage{enumitem}
\usepackage{hyperref}
\usepackage{fp-eval}
\usepackage{esint}
\usepackage{subfigure}
\usepackage{caption}
\usepackage{mathabx}
\usepackage{booktabs}
\usepackage{todonotes}
\usepackage{subdepth}
\definecolor{darkblue}{rgb}{0,0,.7}
\hypersetup{colorlinks=true,linkcolor=darkblue,citecolor=darkblue}

\newlist{alphenum}{enumerate}{1}
\setlist[alphenum]{fullwidth,label={(\alph*)}}

\allowdisplaybreaks[0]
\theoremstyle{definition}
\newtheorem{theorem}{Theorem}[section]

\newtheorem{remark}[theorem]{Remark}
\newtheorem{lemma}[theorem]{Lemma}

\numberwithin{figure}{section}
\numberwithin{table}{section}
\numberwithin{equation}{section}

\newcommand{\dx}{\, d\boldsymbol x}
\newcommand{\ds}{\, ds}

\newcommand{\vecb}[1]{\boldsymbol{#1}}

\newcommand{\Pd}{P^{\mathrm{disc}}}
\newcommand{\Hdiv}{\vecb{H}(\mathrm{div})}

\begin{document}
\date{\today}

%% TITLE %%%%%%%%%%%%%%%%%%%%%%%%%%%%%%%%%%%%%%%%%%%%%%%
\title[Inf-sup stabilized Scott--Vogelius pairs by Raviart--Thomas enrichment]{Inf-sup stabilized Scott--Vogelius pairs on general simplicial grids by Raviart--Thomas enrichment}

%% AUTHORS %%%%%%%%%%%%%%%%%%%%%%%%%%%%%%%%%%%%%%%%%%%%%
\author{Volker~John}
\address{Weierstrass Institute for Applied Analysis and Stochastics,
Leibniz Institute in Forschungsverbund Berlin e.V. (WIAS), Mohrenstr.~39, 10117 Berlin, Germany and
Freie Universit\"at of Berlin,
Department of Mathematics and Computer Science,
Arnimallee 6, 14195 Berlin, Germany}
\email{john@wias-berlin.de}
\author{Xu~Li}
\address{School of Mathematics, Shandong University, Jinan 250100, China}
\email{xulisdu@126.com}
\author{Christian~Merdon}
\address{Weierstrass Institute for Applied Analysis and Stochastics,
Leibniz Institute in Forschungsverbund Berlin e.V. (WIAS), Mohrenstr.~39, 10117 Berlin, Germany}
\email{merdon@wias-berlin.de}
\author{Hongxing~Rui}
\address{School of Mathematics, Shandong University, Jinan 250100, China}
\email{hxrui@sdu.edu.cn}
\thanks{The second author was supported by the China Scholarship Council (No.~202106220106) and the National Natural Science Foundation of China (No.~12131014). The third author gratefully acknowledges the funding by the German Science Foundation (DFG) within the project ``ME 4819/2-1". The fourth author was supported by the National Natural Science Foundation of China (No.~12131014).}

\begin{abstract}
  This paper considers the discretization of the Stokes equations with Scott--Vogelius pairs of finite element
  spaces on arbitrary shape-regular simplicial grids. A novel way of stabilizing these pairs with respect to the discrete inf-sup condition is proposed
  and analyzed. The key idea consists in enriching the continuous polynomials of order $k$ of the Scott--Vogelius
  velocity space with appropriately chosen and explicitly given Raviart--Thomas bubbles. This approach is
  inspired by [Li/Rui, IMA J. Numer. Anal, 2021], where the case $k=1$ was studied. The proposed method is
  pressure-robust, with optimally converging $\vecb{H}^1$-conforming velocity and a small
  $\vecb{H}(\mathrm{div})$-conforming correction rendering the full velocity divergence-free.
  For $k\ge d$, with $d$ being the dimension, the method is parameter-free.
  Furthermore, it is shown that  the additional degrees of freedom
  for the Raviart--Thomas enrichment and also all non-constant pressure degrees of freedom
  can be condensated, effectively leading to a pressure-robust, inf-sup stable, optimally convergent
  $\vecb{P}_k \times P_0$ scheme. Aspects of the implementation are discussed and numerical studies
  confirm the analytic results.
\end{abstract}


\maketitle

\section{Introduction}

The research on divergence-free schemes for incompressible flow equations is a very active
field of research as these schemes have desirable properties with respect to mass conservation
and other structure preservation features. A very important feature is
pressure-robustness \cite{LM:2016, JLMNR:2017}, which guarantees that
the balancing of gradient forces
by the pressure gradient is correctly transferred from the continuous problem to
the discrete problem,
such that the discrete velocity is zero whenever the right-hand
side force is a gradient. The same holds for the balancing of the irrotational part of
the material derivative in time-dependent Navier--Stokes flows \cite{SMAI-JCM_2019__5__89_0}.
It was shown that non-pressure-robust schemes can lead to discrete velocity solutions
that have errors which scale with the inverse viscosity \cite{LM2016,JLMNR:2017}, or even to suboptimal convergence rates in
time-dependent Stokes problems \cite{Linke_2019}. Also optimal estimates for more complicated flow problems
seem to benefit from pressure-robustness \cite{Schroeder_2017,Schroeder_2018, ABBGLM2021,BDV2021,HAN2021113365,GJN21}.

It should be also noted that the set of divergence-free schemes and
the set of pressure-robust schemes are not subsets of each other, as there are non-divergence-free
schemes that can be made pressure-robust by a reconstruction operator technique \cite{LM:2016,LLMS2017}.
On the other hand, there
are divergence-free methods that are not necessarily pressure-robust, like virtual element methods
without a proper right-hand side discretization, see \cite{FM2020,VEMprobust2021}.
In practice, many non-divergence-free schemes are used together with grad-div stabilization.
This technique reduces or even removes the explicit dependence of the constants in error bounds on
inverse powers of the viscosity, e.g., see \cite{DGJN18} for the evolutionary Navier--Stokes equations.
However, it possesses also certain drawbacks, as it introduces a user-chosen parameter and
it is not mass-conservative.

This paper focuses on divergence-free inf-sup stable schemes, where the notion `divergence-free' will be always used
in the sense of `weakly divergence-free', i.e., the divergence is zero in the sense of $L^2(\Omega)$. To this end, the stationary Stokes equations in a bounded domain
$\Omega \subset \mathbb R^d$, $d\in\{2,3\}$, are considered
\begin{equation}\label{eq:stokes}
\begin{array}{rcll}
-\nu \Delta \vecb{u} +\nabla p & = & \vecb{f} & \mbox{in } \Omega,\\
\mathrm{div} (\vecb{u}) & = & 0 &  \mbox{in } \Omega,\\
\vecb{u} & = & \vecb{0} & \mbox{on }\partial \Omega.
\end{array}
\end{equation}
The boundary $\partial\Omega$ of $\Omega$ is assumed to be polyhedral and Lipschitz continuous. Problem \eqref{eq:stokes}
is already written in a dimensionless form, where the kinematic viscosity $\nu\in\mathbb R$ with $\nu>0$ and the
forces $\vecb{f}$ are the given data. The unknown functions are the velocity field $\vecb{u}$ and the pressure $p$.
System \eqref{eq:stokes} is transferred in the usual way to a weak formulation, which reads:
Find $(\vecb{u},p) \in \vecb{V}\times Q := \vecb{H}_0^1(\Omega) \times L_0^2(\Omega)$ such that
\begin{equation}\label{eq:stokes_weak}
\begin{array}{rcll}
(\nu \nabla \vecb{u},\nabla \vecb{v}) - (\mathrm{div} (\vecb{v}),p) & = & (\vecb{f},\vecb{v}) & \forall\ \vecb{v}\in \vecb{V},\\
(\mathrm{div} (\vecb{u}),q) & = & 0 &  \forall\ q\in Q.
\end{array}
\end{equation}
Here, $\vecb{H}_0^1(\Omega)$ is the vector-valued Sobolev space of functions where each component is in $H^1(\Omega)$ and
has a vanishing trace on $\partial \Omega$, $ L_0^2(\Omega)$ is the space of functions from $L^2(\Omega)$ with
vanishing integral mean value, $(\cdot,\cdot)$ denotes the inner product in $L^2(\Omega)$ and $\vecb{L}^2(\Omega)$,
and it is assumed that $\vecb{f} \in \vecb{L}^2(\Omega)$. By the theory of linear saddle point problems, using the
so-called inf-sup condition, it is known that \eqref{eq:stokes_weak} possesses a unique solution, e.g.,
see \cite{GiraultRaviart:1986,Joh16}.

Since the inf-sup stability is a necessary condition for the well-posedness of a linear saddle point
problem, the construction of classical Galerkin finite element schemes traditionally focused on the satisfaction of
a discrete counterpart of this condition. Let $\vecb{V}_h$ and $Q_h$ denote finite element velocity
and pressure spaces, respectively, then the Galerkin method seeks $(\vecb{u}_h,p_h)
\in \vecb{V}_h\times Q_h$ such that
\begin{equation}\label{eq:stokes_fe}
\begin{array}{rcll}
(\nu \nabla \vecb{u}_h,\nabla \vecb{v}_h) - (\mathrm{div} (\vecb{v}_h),p_h) & = & (\vecb{f},\vecb{v}_h) & \forall\ \vecb{v}_h\in \vecb{V}_h,\\
(\mathrm{div} (\vecb{u}_h),q_h) & = & 0 &  \forall\ q_h\in Q_h.
\end{array}
\end{equation}
It turned out that the goal of satisfying a discrete inf-sup condition was often achieved only by relaxing
the divergence constraint, i.e., the mass conservation is, often by far, not satisfied exactly, e.g., see \cite{JLMNR:2017}.
Given a finite element velocity space $\vecb{V}_h$, then the properties of inf-sup stability and mass
conservation lead in fact to different requirements: for the satisfaction of the discrete inf-sup condition,
$Q_h$ should be sufficiently small and for the exact satisfaction of the divergence constraint, $Q_h$ should
be sufficiently large, compare \cite[Rem.~3.56]{Joh16}.

The starting point of the method proposed and analyzed in the current paper is the pair of spaces
$\vecb{V}_h \times Q_h = \vecb{P}_k \times \Pd_{k-1}$, $k\ge 1$, on simplicial triangulations,
where $\vecb{P}_k$ is the space of
continuous and piecewise polynomial vector-valued functions with polynomial degree $k$ and $\Pd_{k-1}$
is the space of piecewise polynomial functions with polynomial degree $k-1$. As usual, the notation
does not contain the facts that $\vecb{P}_k$ is intersected with $\vecb{V}$ and
$ \Pd_{k-1}$ with $Q$. For $k\geq d$, this pair of spaces is also known as Scott--Vogelius pair \cite{SV:1983,SV1985,Arnold1992,Zhang:2005,Neilan2015}. It is
$\mathrm{div} (\vecb{P}_k) \subseteq \Pd_{k-1}$, so that if $\vecb{u}_h$ is a velocity solution of \eqref{eq:stokes_fe},
one can choose $q_h = \mathrm{div}(\vecb{u}_h)$ in the discrete divergence constraint, leading to
\begin{equation}\label{eq:u_h_div_free}
(\mathrm{div} (\vecb{u}_h), \mathrm{div} (\vecb{u}_h)) = \|\mathrm{div} (\vecb{u}_h)\|^2 = 0,
\end{equation}
which means that $\vecb{u}_h$ is divergence-free. This property is our major motivation for studying
Scott--Vogelius pairs. However, it is well known that the discrete inf-sup condition for these pairs
is valid only for certain classes of meshes and if $k$ is sufficiently large, e.g., see \cite{Zhang:2011b,GS2018,GS2019}. In particular, for $d=3$, which is the interesting case in applications, the smallest
value is $k=3$ for so-called barycentric-refined meshes \cite{Zhang:2005}.
For $k=2$ there is a related method on Powell-Sabin tetrahedral grids with an implicitly defined pressure subspace \cite{Zhang:2011}.

This paper addresses the cases where the inf-sup stability is not given. For such cases, an enrichment
of the discrete velocity space is proposed that consists of local functions, locally very few, which leads to discrete velocity fields that satisfy
\eqref{eq:u_h_div_free} and to pressure-robust velocity error estimates. The last two properties are in
contrast to previous proposals, like the family of Bernardi--Raugel elements that add higher-order polynomial bubbles \cite{BR1985}.
Moreover, the proposed method is easy to implement, in particular it does not involve any face integrals.

\iffalse
The inf-sup stability is a desirable property as it ensures unique solvability and optimal pressure
approximations and it is in fact the main reason why many schemes were suggested that
relax the divergence constraint in favor of inf-sup stability \cite{JLMNR:2017}.
Indeed, in general inf-sup stability for the $\vecb{P}_k \times P_{k-1}$ pairing
(where $\vecb{P}_k$ are continuous vector-valued polynomials) only
holds on certain classes of meshes and certain restrictions on $k$, see e.g.\
\cite{SV:1983,Zhang:2011,Zhang:2011b}. To obtain inf-sup stability in other situations
there are essentially two possibilities: reducing the pressure space, or enriching the velocity
space. The first possibility leads to stable elements like the $\vecb{P}_2 \times P_0$ element
in two dimensions or the famous Taylor--Hood element, while the other possibility compromises
, e.g., the Bernardi--Raugel family or other elements that add higher-order polynomial bubbles \cite{BR1985}.
Both approaches relax the divergence constraint in the sense that the condition
\begin{align*}
  (\mathrm{div} (\vecb{v}_h), q_h) = 0 \text{ for all } q_h \in Q_h
\end{align*}
does not imply \(\mathrm{div} (\vecb{v}_h) = 0\) for such velocity ansatz functions \(\vecb{v}_h\)
with \(\mathrm{div} (\vecb{v}_h) \not\in Q_h\). These functions than also interact with
the irrotational forces in the momentum balance, opposite to real divergence-free functions.
This is in fact the crucial structural property behind pressure-robust schemes.
\fi

In the literature there are meanwhile several approaches for constructing divergence-free pairs of finite
element spaces. In  \cite{10.2307/24488200, GN2018, SH2019}, $\vecb{H}^1$-conforming
functions are utilized for enriching the velocity space, which are however non-standard and not available
in many finite element codes.
Another strategy employs the Stokes complex of lowest regularity where the velocity is searched only in
$\Hdiv$, e.g., see \cite{CKS2007, Wang2007,Anaya_2019}.
This approach  leads to the nowadays quite popular hybrid discontinuous Galerkin methods (HDG)
\cite{CCS2006,Lehrenfeld_2016}.

In this paper we propose and analyze an approach for enriching Scott--Vogelius pairs by $\vecb{H}(\mathrm{div})$-conforming functions.
The main idea of the method is inspired by \cite{li2021low}, where the lowest-order case $k=1$ is studied. Therein a novel scheme was suggested which
preserves the features of the $\vecb{H}^1$-conforming formulation while using a simple lowest-order Raviart--Thomas enrichment.
It turns out that there are for $k\ge2$ new aspects in the construction of the method and
additional difficulties in their numerical analysis.
The standard polynomial space $\vecb{P}_k$ is enriched with an $\vecb{H}(\mathrm{div})$-conforming subspace
of specially chosen but standard Raviart--Thomas functions, which ensures that the discrete inf-sup condition holds and that the
divergence constraint is satisfied exactly on general unstructured shape-regular simplicial meshes.
In particular no assumption on non-singular vertices is needed.
The added $\vecb{H}(\mathrm{div})$-conforming part vanishes in the limit $h \rightarrow 0$, while the $\vecb{H}^1$-conforming part
is an approximation of the velocity of optimal order.  If $k < d$, the lowest-order
Raviart--Thomas space is involved, while for the higher order cases $k \geq d$, only interior non-divergence-free Raviart--Thomas bubbles of order $k-1$ are used for the enrichment.
These additional degrees of freedom ensure the inf-sup stability of the scheme and lead to a pressure-robust
$\vecb{H}^1$-conforming solution, which
can be turned into a divergence-free solution by adding the $\Hdiv$-conforming Raviart--Thomas correction.

A subtle point of the proposed method is that the space $\vecb{P}_k$, $k\ge 2$, and the enrichment space might
have a non-zero intersection on certain triangulations, for which we
also provide an example. Nevertheless, it is ensured that the method selects a unique discrete solution.
For $k<d$ there is a stabilization on the lowest-order Raviart--Thomas part, which involves a parameter and essentially penalizes this part of the enrichment component,
as in the original scheme for $k=1$ from \cite{li2021low}.
For any higher-order Raviart--Thomas part, no stabilization is needed, hence the scheme is parameter-free for $k\geq d$,
and the uniqueness follows instead from the bijectivity of the divergence operator
with respect to the enrichment space and the divergence constraint.

The final result of the paper concerns a condensed method in the spirit of other divergence-free schemes \cite{Lehrenfeld2010,Veiga17,Wei_VEM2021}.
Indeed, it is possible to statically condensate the additional Raviart--Thomas velocity degrees of freedom and
the higher order pressure degrees of freedom, effectively resulting in  a $\vecb{P}_k \times P_0$ scheme.
The condensed scheme preserves all the properties of the full
scheme and the missing degrees of freedom can be restored by a cheap post-processing. Behind this transformation of the problem is a
correction operator that maps each conforming polynomial test function to a Raviart--Thomas correction such that the divergence of their sum
is a piecewise constant.

The paper is structured as follows. Section~\ref{sec:main_concept} explains the main
concept of the scheme and explicitly constructs the needed Raviart--Thomas subspaces for enrichment up to a certain order.
The full discrete problem is presented in Section~\ref{sec:full_scheme} and the mechanisms behind
the uniqueness of the solution are discussed. In Section~\ref{sec:infsup} it is  shown that the enrichment spaces ensure
inf-sup stability, i.e., the missing ingredient for unique solvability.
Section~\ref{sec:apriori_estimates} derives optimal a priori error estimates
for the velocity and the pressure via standard arguments. The derivation of the condensed scheme and the post-processing
technique for recovering the full solution are
presented in Section~\ref{sec:reduced_scheme}. Some numerical studies in Section~\ref{sec:numerics}
confirm the results in two and three dimensions and for different polynomial orders $k$.
The paper closes with a summary and outlook.

\section{Main concept}
\label{sec:main_concept}
This section explains the main concept and the design of the Raviart--Thomas enrichment spaces used in the suggested method.

\subsection{Notation and preliminaries}
\label{sec:notation}
Consider a regular triangulation \(\mathcal{T}\) of the domain \(\Omega\) with nodes \(\mathcal{N}\) and facets \(\mathcal{F}\).
The set of all interior faces is denoted by $\mathcal{F}^0$.
Denote by $h_T$ the diameter of elements $T\in \mathcal{T}$ and define $h:=\max_{T\in\mathcal{T}}h_T$
as well as the piecewise constant function $h_\mathcal{T}$ via $h_\mathcal{T}|_T := h_T$ for all $T \in \mathcal{T}$.

The space of scalar-valued polynomials of order $k$ on a subdomain $\omega$ is denoted by \(P_k(\omega)\) and is written in bold for
vector-valued polynomial spaces.
The subspace of divergence-free functions is denoted by
\begin{align*}
  \vecb{V}_0 := \left\lbrace \vecb{v} \in \vecb{V} : \mathrm{div} (\vecb{v}) = 0 \right\rbrace.
\end{align*}

Also define
\begin{align*}
{P}_k(\mathcal{T}):=\left\{q_h\in H^1(\Omega): q_h|_T\in P_k(T) \text{ for all } T\in \mathcal{T}\right\},
\end{align*}
and
\begin{align*}
{{P}}_k^{\mathrm{disc}}(\mathcal{T}):=\left\{q_h\in L^2(\Omega): q_h|_T\in P_k(T) \text{ for all } T\in \mathcal{T}\right\}.
\end{align*}

Furthermore, the space of Raviart--Thomas functions on a cell $T \in \mathcal{T}$ is given by
\begin{align*}
  \vecb{RT}_{k}(T) := \left\lbrace \vecb{v} \in \vecb{L}^2 (T) :\, \exists \vecb{p} \in \vecb{P}_k(T), q \in {P}_k(T), \ \vecb{v}|_T(\vecb{x}) = \vecb{p}(\vecb{x}) + q(\vecb{x}) \vecb{x} \right\rbrace
\end{align*}
and it is the building block of the global space
\begin{align*}
  \vecb{RT}_{k}(\mathcal{T}) := \left\lbrace \vecb{v} \in \vecb{H}(\mathrm{div},\Omega) : \forall T \in \mathcal{T} \, \vecb{v}|_T \in \vecb{RT}_{k}(T) \right\rbrace.
\end{align*}
The Raviart--Thomas subspace of interior bubble functions reads
\begin{align*}
\vecb{RT}_k^{\mathrm{int}}(\mathcal{T}):=\left\lbrace\vecb{v}\in\vecb{RT}_k(\mathcal{T}): \vecb{v}\cdot\vecb{n}_T|_{\partial T}=0 \text{ for all } T\in\mathcal{T}\right\rbrace,
\end{align*}
where $\vecb{n}_T$ is the outer unit normal vector along $\partial T$.
The subspace of $\vecb{RT}_k^{\mathrm{int}}(\mathcal{T})$ consisting of divergence-free functions and its arbitrary but fixed complement space are denoted by $\vecb{RT}_{k,0}^{\mathrm{int}}(\mathcal{T})$ and $\widetilde{\vecb{RT}}_k^{\mathrm{int}}(\mathcal{T})$, respectively, i.e.,
\begin{align*}
\vecb{RT}_k^{\mathrm{int}}(\mathcal{T})=\vecb{RT}_{k,0}^{\mathrm{int}}(\mathcal{T})\oplus \widetilde{\vecb{RT}}_k^{\mathrm{int}}(\mathcal{T}).
\end{align*}
The only divergence-free function in $\widetilde{\vecb{RT}}_k^{\mathrm{int}}(\mathcal{T})$ is the zero function, which implies that the divergence operator on this space is injective. Note that $\widetilde{\vecb{RT}}_k^{\mathrm{int}}(\mathcal{T})$ is not unique in general cases. Here we require that its local spaces $\widetilde{\vecb{RT}}_k^{\mathrm{int}}(T), T\in \mathcal{T}$, have the same structure in the sense that all of them are connected to a same reference space via Piola's transformation (see e.g.\ \cite[Eq.~2.1.69]{BBF:book:2013}). This is natural for $\widetilde{\vecb{RT}}_k^{\mathrm{int}}(\mathcal{T})$ since it is characterized by the normal trace and the divergence, which are preserved (in a scaled meaning) by Piola's transformation.

The subspace of $P_{k}^{\mathrm{disc}}(\mathcal{T})$ consisting of elementwise zero-mean functions is defined as
\begin{equation}\label{eq:def_tildepk_disc}
\widetilde{P}_{k}^{\mathrm{disc}}(\mathcal{T}):=\left\{q_h\in {P}_{k}^{\mathrm{disc}}(\mathcal{T}): (q_h, 1)_T=0 \text{ for all } T\in\mathcal{T}\right\}.
\end{equation}
Note that $\vecb{RT}_0^{\mathrm{int}}(\mathcal{T})=\widetilde{\vecb{RT}}_0^{\mathrm{int}}(\mathcal{T})=\left\{\vecb{0}\right\}$ and $\widetilde{P}_{0}^{\mathrm{disc}}(\mathcal{T})=\lbrace 0 \rbrace$.

Throughout this paper, for any (scalar or vector-valued) finite element space $\mathcal{S}$
(with or without an argument like $\mathcal{T}$), its local version on each element $T$
is denoted by $\mathcal{S}(T)$ if not specially indicated. The symbol $\pi_\mathcal{S}$ ($\pi_{\mathcal{S}(T)}$) is used to denote the $L^2$ projection operator onto $\mathcal{S}$ ($\mathcal{S}(T)$, respectively).

\subsection{Enrichment approach}
\label{sec:enrichment_principle}
For a given $k \geq 1$, we define
the $\vecb{H}^1$-conforming velocity ansatz
space of piecewise vector-valued polynomials
\begin{align*}
	\vecb{V}_h^{\mathrm{ct}} := \vecb{P}_k(\mathcal{T}) \cap \vecb{V}
\end{align*}
and the desired pressure space
\begin{align*}
  Q_h := {P}_{k-1}^{\mathrm{disc}}(\mathcal{T}) \cap Q.
\end{align*}

In general, without further assumptions on the mesh or $k$,
one has to expect a violation of the inf-sup stability.
However, we can always assume the existence of a sufficiently small auxiliary piecewise
pressure subspace $\widehat{Q}_h = (\bigcup_{T \in \mathcal{T}} \widehat{Q}_h(T)) \cap Q$
with $\widehat{Q}_h(T) \subseteq P_{k-1}(T)$ (it can be sometimes only the zero space)
such that \((\vecb{V}_h^{\mathrm{ct}}, \widehat{Q}_h)\) is inf-sup stable.
Consider now the $L^2$ orthogonal split of \(Q_h\) into \(Q_h=\widehat{Q}_h\oplus_{L^2}\widehat{Q}_h^{\perp}\).
The main motivation for finding the enrichment space $\vecb{V}_h^\mathrm{R}$ is twofold:
on the one hand, $\vecb{V}_h^\mathrm{R}$ should fix the potential spurious pressure modes in $\widehat{Q}_h^\perp$;
on the other hand, $\mathrm{div}(\vecb{V}_h^\mathrm{R}) \subseteq Q_h$ guarantees that the
discrete velocity is still divergence-free.

In fact, we suggest to select a subspace of $\vecb{RT}_{k-1}$
such that either
\begin{align}\label{eqn:bijection_property}
    \mathrm{div} : \vecb{V}_h^\mathrm{R} \rightarrow \widehat{Q}_h^{\perp}
   \quad \text{is bijective}\text{ for } k\geq d,
\end{align}
or
\begin{align}\label{eqn:bijection_property2}
   \mathrm{div} : \vecb{V}_h^\mathrm{R} \rightarrow \widehat{Q}_h^{\perp}
   \quad \text{is surjective} \text{ for } k<d.
\end{align}
Eventually, it is shown that an inf-sup stable scheme for
the velocity space \(\vecb{V}_h := \vecb{V}_h^{\mathrm{ct}} \times \vecb{V}_h^\mathrm{R}\) and the full pressure space \(Q_h\) can be established,
such that its solution \(\vecb{u}_h = (\vecb{u}_h^{\mathrm{ct}}, \vecb{u}_h^{\mathrm{R}}) \in \vecb{V}_h^{\mathrm{ct}} \times \vecb{V}_h^\mathrm{R}\) is divergence-free in the sense that
\(\mathrm{div}(\vecb{u}_h^{\mathrm{ct}} + \vecb{u}_h^{\mathrm{R}}) = 0\).

Of course, the choice of $\widehat{Q}_h$ is not unique in general. Properties \eqref{eqn:bijection_property} and \eqref{eqn:bijection_property2} are the fundamental principles for selecting the spaces.
Tables~\ref{tab:spaces2D} and \ref{tab:spaces3D} list some guaranteed
inf-sup stable pairs \((\vecb{V}_h^{\mathrm{ct}}, \widehat{Q}_h)\), implied, e.g.,\ by \cite[Sec. 8.6-8.7]{BBF:book:2013},
for different $k$ in two and three dimensions, respectively.
The other columns in these tables state the expected dimension of the enrichment spaces
and some hint on how the enrichment functions can be chosen in each case.
Their detailed construction is explained in the next subsection
where it is also
shown that the divergence operator between $\vecb{V}_h^\mathrm{R}$ and $\widehat{Q}_h^\perp$
is bijective (except for $k< d$).
Generalizing the examples from Tables~\ref{tab:spaces2D} and \ref{tab:spaces3D}, the space
$\widehat{Q}_h$ can be chosen as
\begin{equation}\label{eq:def_qh}
\begin{aligned}
\widehat{Q}_h:=\begin{cases}
\left\{\vecb{0}\right\} \quad & k<d,\\
{P}_{k-d}^{\mathrm{disc}}(\mathcal{T}) \cap Q \quad &  k\geq d,\\
\end{cases}
\end{aligned}
\end{equation}
and the corresponding Raviart--Thomas enrichment space $\vecb{V}_h^\mathrm{R}$ is chosen as
\begin{equation}\label{eq:def_vhr}
\begin{aligned}
\vecb{V}_h^\mathrm{R}:=\begin{cases}
\vecb{RT}_{0}(\mathcal{T}) \cap \vecb{H}_0(\mathrm{div},\Omega) \quad & k=1,\\
(\vecb{RT}_{0}(\mathcal{T}) \cap \vecb{H}_0(\mathrm{div},\Omega)) \oplus \vecb{RT}_{1}^{\text{int}}(\mathcal{T})\quad & k=2, d=3,\\
\left\{\vecb{v}_{h}\in\widetilde{\vecb{RT}}_{k-1}^{\text{int}}(\mathcal{T}): \mathrm{div}(\vecb{v}_h)\in \widehat{Q}_h^{\perp}\right\}\quad & k\geq d.
\end{cases}
\end{aligned}
\end{equation}
Here, \(\vecb{H}_0(\mathrm{div},\Omega)\) denotes the space of functions from \(\vecb{H}(\mathrm{div},\Omega)\)
with zero normal trace along \(\partial \Omega\).
The local dimension of $\vecb{V}_h^\mathrm{R}$ is
\begin{align*}
\mathrm{dim}\left(\vecb{V}_{h}^{\mathrm{R}}(T)\right)=\begin{cases}
d+1 \quad & k=1,\\
\mathrm{dim}\left(P_{k-1}(T)\right)-\mathrm{dim}\left(P_{k-2}(T)\right)=k   \quad & k\geq 2,d=2,\\
7   \quad & k=2, d=3,\\
\mathrm{dim}\left(P_{k-1}(T)\right)-\mathrm{dim}\left(P_{k-3}(T)\right)=k^2   \quad & k\geq 3,d=3.\\
\end{cases}
\end{align*}

\begin{table}
  \caption{\label{tab:spaces2D}Enrichment spaces for $d=2$ and $k=1,2,3,4$.}
  \begin{tabular}{ccccl}
  k & $\widehat{Q}_h(T)$ &   $\mathrm{dim}(\widehat{Q}_h^{\perp}(T))$ & $\mathrm{dim}(\vecb{V}_h^\mathrm{R}(T))$ \\
  \hline
  1 & $\lbrace 0 \rbrace$ &  1 & 3 & full $\vecb{RT}_0(T)$\\
  2 & $P_0(T)$            &  2 & 2 & full $\vecb{RT}^\text{int}_1(T)$\\
  3 & $P_1(T)$            &  3 & 3 & from $\widetilde{\vecb{RT}}^\text{int}_2(T)$\\
  4 & $P_2(T)$            &  4 & 4 & from $\widetilde{\vecb{RT}}^\text{int}_3(T)$\\
\end{tabular}
\end{table}
\begin{table}
  \caption{\label{tab:spaces3D}Enrichment spaces for $d=3$ and $k=1,2,3$.}
  \begin{tabular}{ccccl}
  k & $\widehat{Q}_h(T)$ &$\mathrm{dim}(\widehat{Q}_h^{\perp}(T))$ & $\mathrm{dim}(\vecb{V}_h^\mathrm{R}(T))$ \\
  \hline
  1 & $\lbrace 0 \rbrace$ &  1 & 4 & full $\vecb{RT}_0(T)$\\
  2 & $\lbrace 0 \rbrace$ &  4 & 7 & full $\vecb{RT}_0(T)$ \& full $\vecb{RT}^\text{int}_1(T)$\\
  3 & $P_0(T)$            &  9 & 9 & from $\widetilde{\vecb{RT}}^\text{int}_2(T)$\\
\end{tabular}
\end{table}

Throughout this paper, the analysis for $k<d$ cases is based on the spaces by \eqref{eq:def_vhr}, while for $k\geq d$ cases we only require $\widehat{Q}_h(T)\supseteq {P}_0(T)$ and further $\vecb{V}_h^\mathrm{R}\times\widehat{Q}_h^{\perp}\subseteq \widetilde{\vecb{RT}}_{k-1}^{\mathrm{int}}(\mathcal{T})\times\widetilde{P}_{k-1}^{\mathrm{disc}}(\mathcal{T})$ satisfies \eqref{eqn:bijection_property}.
Indeed, from the separation of $\vecb{H}(\mathrm{div})$-conforming finite element spaces, e.g., as in \cite[\S\,2.2.4]{Lehrenfeld2010} and \cite[\S\,5.2-5.3]{Zaglmayr2006}, one can verify that
\(\mathrm{div}: \vecb{RT}_k^{\mathrm{int}}(\mathcal{T})\rightarrow \widetilde{P}_{k}^{\mathrm{disc}}(\mathcal{T})\) is a surjective operator and thus
\(\mathrm{div}: \widetilde{\vecb{RT}}_k^{\mathrm{int}}(\mathcal{T})\rightarrow \widetilde{P}_{k}^{\mathrm{disc}}(\mathcal{T})\) is a bijective operator.
Also note that
$\widetilde{\vecb{RT}}_1^{\mathrm{int}}(\mathcal{T})=\vecb{RT}_1^{\mathrm{int}}(\mathcal{T})$.

\begin{remark}[On over-enrichment]
  There may be cases where a larger space $\widehat{Q}_h$ and hence
  a smaller enrichment space is sufficient.
  This is certainly also connected to the mesh properties.
To give one example:
if the mesh is a barycentric refinement of some given mesh, then inf-sup stability is already ensured
for the full pressure space, i.e.,\ $\widehat{Q}_h = Q_h$ for $k \geq d$ is a valid choice,
which results in the Scott--Vogelius finite element family \cite{SV:1983,Arnold1992,Zhang:2005}.
Also in case of unstructured grids with non-singular vertices, positive results for
$k \geq 3$ in two dimensions are known \cite{GS2018,GS2019}.
From this perspective our schemes can be seen as a stabilization of the
Scott--Vogelius family that ensures inf-sup stability on general shape-regular meshes,
even with singular vertices, that works for any $k \geq 1$.

Moreover, it turns out that a potential
over-enrichment is not really an issue in practice.
In fact, Section~\ref{sec:reduced_scheme} explains how all additional enrichment degrees of freedom
can be effectively condensated. By `over-enrichment', the pressure space can be even decreased to
piecewise constants without compromising any qualitative properties of the scheme.
\end{remark}

\subsection{Explicit constructions of the enrichment space}
\label{sec:enrichment_space_choices}

This subsection discusses explicit element-wise constructions of basis functions of \(\vecb{V}_h^\mathrm{R}\) for
\(k = 2, \ldots, 6-d\) beyond the $k=1$ case in \cite{li2021low}. To fix local enumerations, consider a simplex \(T \in \mathcal{T}\)
with faces \(F_j\) and their respective opposing vertex \(\vecb{P}_j\)
with nodal basis function \(\varphi_j\).
The lowest-order Raviart--Thomas functions read
\begin{align*}
  \vecb{\psi}^{\mathrm{RT}_0}_j := \frac{1}{d \lvert T \rvert} \left(\vecb{x} - \vecb{P}_j\right)
  \quad \text{such that} \quad \int_{F_k} \vecb{\psi}^{\mathrm{RT}_0}_j \cdot \vecb{n}_T \textit{ds} = \delta_{jk},
\end{align*}
for \(j,k = 1,\ldots,d+1\).
By a multiplication of these basis functions with their opposite nodal basis functions, one obtains
the interior Raviart--Thomas functions
\begin{align*}
  \vecb{\psi}^{\mathrm{RT}_1}_j := \varphi_j \vecb{\psi}^{\mathrm{RT}_0}_j \in \widetilde{\vecb{RT}}_{1}^{\text{int}}(T).
\end{align*}
Note that only $d$ of them are linearly
independent. Indeed it holds $\vecb{\psi}^{\mathrm{RT}_1}_{d+1} = - \sum_{j=1}^{d} \vecb{\psi}^{\mathrm{RT}_1}_j$.

\begin{lemma}[Explicit design of $\vecb{V}_h^\mathrm{R}$ for $d=2$]\label{lem:vhr_2d}
  On any $T \in \mathcal{T}$, the following statements hold:
  \begin{itemize}
    \item[(a)] The functions \(\lbrace \vecb{\psi}^{\mathrm{RT}_{1}}_j \rbrace_{j = 1,2} \subset \widetilde{\vecb{RT}}_{1}^{\text{int}}(T)\) are linearly independent
    and it holds
  \begin{align*}
    \int_T \mathrm{div}(\vecb{\psi}^{\mathrm{RT}_1}_j) \dx = 0.
  \end{align*}
For $\vecb{V}_h^\mathrm{R}(T) :=  \mathrm{span} \lbrace \vecb{\psi}^{\mathrm{RT}_1}_1, \vecb{\psi}^{\mathrm{RT}_1}_2 \rbrace=\widetilde{\vecb{RT}}_{1}^{\text{int}}(T)$
the restricted divergence operator
\begin{align*}
   \mathrm{div} : \vecb{V}_h^\mathrm{R}(T) \rightarrow \widehat{Q}_h^\perp(T) \quad \text{for } \widehat{Q}_h(T) = P_0(T)
   \quad \text{is bijective.}
\end{align*}
\item[(b)] The functions
\begin{align*}
  \vecb{\psi}^{\mathrm{RT}_2}_{j} :=
  \left(5 \varphi_j - 2\right) \vecb{\psi}^{\mathrm{RT}_1}_j \in \widetilde{\vecb{RT}}_{2}^{\text{int}}(T)
  \quad \text{for } j =1,2,3,
\end{align*}
are linearly independent
and it holds
\begin{equation*}
  \int_T \mathrm{div}(\vecb{\psi}^{\mathrm{RT}_2}_{j}) \varphi_k \dx = 0 \quad \text{for } j,k = 1,2,3.
\end{equation*}
For $\vecb{V}_h^\mathrm{R}(T) :=  \mathrm{span} \lbrace \vecb{\psi}^{\mathrm{RT}_2}_j : j =1,2,3 \rbrace$
the restricted divergence operator
\begin{align*}
  \mathrm{div} : \vecb{V}_h^\mathrm{R}(T) \rightarrow \widehat{Q}_h^\perp(T) \quad \text{for } \widehat{Q}_h(T) = P_1(T)
  \quad \text{is bijective.}
\end{align*}
\item[(c)] The three auxiliary face-related functions
\begin{align*}
  \vecb{\psi}^{\mathrm{RT}_3}_{j} := \frac{1}{7}\left(7 \varphi_j^2 - 6 \varphi_j + 1\right) \vecb{\psi}^{\mathrm{RT}_1}_j
  \in \widetilde{\vecb{RT}}_{3}^{\text{int}}(T) \quad \text{for } j = 1,2,3,
\end{align*}
and additionally
\begin{multline*}
  \hspace{1cm} \vecb{\psi}^{\mathrm{RT}_3}_{4} := -2\varphi_2\varphi_3\vecb{\psi}^{\mathrm{RT}_1}_2+\frac{2}{45}\left(\vecb{\psi}^{\mathrm{RT}_1}_1+5\vecb{\psi}^{\mathrm{RT}_1}_2\right)\\
      +\frac{1}{70}\left(3\vecb{\psi}^{\mathrm{RT}_2}_{1}+2\vecb{\psi}^{\mathrm{RT}_2}_{2}-3\vecb{\psi}^{\mathrm{RT}_2}_{3}\right)
      \in \widetilde{\vecb{RT}}_{3}^{\text{int}}(T), \hspace{1cm}
\end{multline*}
are linearly independent and it holds
\begin{equation*}
  \int_T \mathrm{div}(\vecb{\psi}^{\mathrm{RT}_3}_{j}) \lambda_h \dx = 0 \quad \text{for } \lambda_h \in P_2(T), j = 1,2,3,4.
\end{equation*}
For $\vecb{V}_h^\mathrm{R}(T) :=  \mathrm{span} \lbrace \vecb{\psi}^{\mathrm{RT}_3}_{j} : j =1,2,3,4 \rbrace$
the restricted divergence operator
\begin{align*}
  \mathrm{div} : \vecb{V}_h^\mathrm{R}(T) \rightarrow \widehat{Q}_h^\perp(T) \quad \text{for } \widehat{Q}_h(T) = P_2(T)
  \quad \text{is bijective.}
\end{align*}
\end{itemize}
\end{lemma}
\begin{proof}[Proof of (a)]
  By simple calculations on the reference element and a dimension argument, one can obtain
  \begin{align}\label{eqn:divmoments_RT1}
    \int_{T} \mathrm{div}({\boldsymbol\psi}^{\mathrm{RT}_1}_{j}) \varphi_k \, d\boldsymbol x = \frac{3 \delta_{jk} - 1}{24}
    \quad \text{for } j,k =1,2,3.
  \end{align}
  Then the linear independence of \(\vecb{\psi}^{\mathrm{RT}_1}_1\) and \(\vecb{\psi}^{\mathrm{RT}_1}_2\) follows from $\mathrm{det}(A)=\frac{1}{192}\neq 0$ with $A=(a_{jk}):=(\int_{T} \mathrm{div}({\boldsymbol\psi}^{\mathrm{RT}_1}_{j}) \varphi_k \, d\boldsymbol x)\in\mathbb{R}^{2\times2}$. In fact, if there exists a vector $\vec{c}=(c_1,c_2)\in\mathbb{R}^2$ which makes $c_1\mathrm{div}({\boldsymbol\psi}^{\mathrm{RT}_1}_{1})+c_2\mathrm{div}({\boldsymbol\psi}^{\mathrm{RT}_1}_{2})=0$, one has $A\vec{c}=0$ and further $\vec{c}=0$, which demonstrates that $\vecb{\psi}^{\mathrm{RT}_{1}}_1$ and $\vecb{\psi}^{\mathrm{RT}_{1}}_2$ are linearly independent (and their divergence also). The second assertion follows from the Gauss theorem and the fact that the constructed functions are normal-trace-free (or alternatively by summing up \eqref{eqn:divmoments_RT1} for
  $k = 1,2,3$). Finally, the bijectivity of the divergence map follows from the orthogonality (to $P_0(T)$) and the linear independence of the divergence of \(\vecb{\psi}^{\mathrm{RT}_1}_1\) and \(\vecb{\psi}^{\mathrm{RT}_1}_2\).

\textit{Proof of (b).}
  Similar calculations as in (a) yield
  \begin{align*}
    \int_{T} \mathrm{div}(\varphi_j {\boldsymbol\psi}^{\mathrm{RT}_1}_j) \varphi_k
    \, d{\boldsymbol x}= \frac{3 \delta_{jk} - 1}{60}
    \quad \text{and (with \eqref{eqn:divmoments_RT1})} \quad
    \int_{T} \mathrm{div}({\boldsymbol\psi}^{\mathrm{RT}_2}_{j}) \varphi_k \, d{\boldsymbol x} = 0,
  \end{align*}
  for $j,k =1,2,3$. The linear independence and the bijectivity of the divergence map
  follow from looking at the higher order divergence moments
  \begin{equation}\label{quadmoments}
    a_{jk} := \int_{T} \mathrm{div}({\boldsymbol\psi}^{\mathrm{RT}_2}_{j}) \chi_k \, d{\boldsymbol x} = \frac{4\delta_{jk}-3}{180} \quad \text{for } j,k = 1,2,3,
  \end{equation}
  for the three face bubbles \(\chi_k = \varphi_{k+1} \varphi_{k-1}\) (the subscript values here should be understood in a circular way). The matrix \(A := (a_{jk})_{j,k=1,2,3}\) is
  nonsingular, since \(\det(A) = -\frac{1}{72900}\).

  \textit{Proof of (c).}
    These statements can be proven in a similar way as in the other cases or simply checked numerically, by computing
    the necessary matrices on the reference domain.
\end{proof}

\begin{lemma}[Explicit design of $\vecb{V}_h^\mathrm{R}$ for $d=3$]
  On any $T \in \mathcal{T}$, the following statements hold:
  \begin{itemize}
    \item[(a)] The functions \(\lbrace \vecb{\psi}^{\mathrm{RT}_{0}}_j \rbrace_{j = 1,2,3,4}\) and
    \(\lbrace \vecb{\psi}^{\mathrm{RT}_{1}}_j \rbrace_{j = 1,2,3} \subset \widetilde{\vecb{RT}}_{1}^{\text{int}}(T)\) are linearly independent.
For $\vecb{V}_h^\mathrm{R}(T) := \mathrm{span} \lbrace \vecb{\psi}^{\mathrm{RT}_0}_j : j=1,2,3,4\rbrace
\oplus \mathrm{span} \lbrace \vecb{\psi}^{\mathrm{RT}_1}_j : j= 1,2,3  \rbrace$
the restricted divergence operator
\begin{align*}
   \mathrm{div} : \vecb{V}_h^\mathrm{R}(T) \rightarrow P_1(T)
   \quad \text{is surjective (but not bijective).}
\end{align*}
    \item[(b)] The functions \(\lbrace \vecb{\psi}^{\mathrm{RT}_{1}}_j \rbrace_{j = 1,2,3} \subset \widetilde{\vecb{RT}}_{1}^{\text{int}}(T)\) and
    \begin{equation*}
      \vecb{\psi}_{(j,k)}^{\mathrm{RT}_{2}}=(6\varphi_{j} - 1) \vecb{\psi}_{k}^{\mathrm{RT}_{1}}+(6\varphi_{k}-1)\vecb{\psi}_{j}^{\mathrm{RT}_{1}} \in \widetilde{\vecb{RT}}_{2}^{\text{int}}(T)
      \quad \text{for } (j,k)\in \mathcal{S},
    \end{equation*}
    with \(\mathcal{S}=\big\{(1,2),(1,3),(1,4),(2,3),(2,4),(3,4)\big\}\) are linearly independent
    and it holds
  \begin{align*}
    \int_T \mathrm{div}(\vecb{\psi}^{\mathrm{RT}_1}_j) \dx & = 0
    \quad \text{for } j= 1,2,3,\\
    \int_T \mathrm{div}(\vecb{\psi}^{\mathrm{RT}_2}_{(j,k)}) \dx & = 0
    \quad \text{for } (j,k)\in \mathcal{S}.
  \end{align*}
For $\vecb{V}_h^\mathrm{R}(T) := \mathrm{span} \lbrace \vecb{\psi}^{\mathrm{RT}_1}_j : j = 1,2,3  \rbrace
\oplus \mathrm{span} \lbrace \vecb{\psi}_{(j,k)}^{\mathrm{RT}_{2}} : (j,k)\in \mathcal{S} \rbrace$
the restricted divergence operator
\begin{align*}
   \mathrm{div} : \vecb{V}_h^\mathrm{R}(T) \rightarrow \widehat{Q}_h^\perp(T) \quad \text{for } \widehat{Q}_h(T) = P_0(T)
   \quad \text{is bijective.}
\end{align*}
\end{itemize}
\end{lemma}
\begin{proof}
The proof of this lemma follows the same arguments as the proof in two dimensions.
\end{proof}

\section{The full scheme and the uniqueness of its solution}
\label{sec:full_scheme}

This section formulates the full scheme
and discusses the uniqueness of the solution.

\subsection{The scheme}
Consider the ansatz space
\begin{align*}
	\vecb{V}_h := \vecb{V}_h^{\mathrm{ct}} \times \vecb{V}_h^\mathrm{R},
\end{align*}
with \(\vecb{V}_h^{\mathrm{ct}} := \vecb{P}_k(\mathcal{T}) \cap \vecb{V}\) and the enrichment space \(\vecb{V}_h^\mathrm{R} \subset \vecb{H}_0(\mathrm{div},\Omega)\)
selected according to the description in Subsection~\ref{sec:enrichment_principle} (hence either \eqref{eqn:bijection_property} or \eqref{eqn:bijection_property2} holds). Given a function $\vecb{v}_h \in \vecb{V}_h^{\mathrm{ct}} \times \vecb{V}_h^\mathrm{R}$
the two components in $\vecb{V}_h^{\mathrm{ct}}$ and $\vecb{V}_h^\mathrm{R}$ are denoted by $\vecb{v}_h^{\mathrm{ct}}$ and $\vecb{v}_h^{\mathrm{R}}$, respectively.
The pressure space is \(Q_h := {{P}}_{k-1}^{\mathrm{disc}}(\mathcal{T}) \cap Q\).

Any $\vecb{v}_h^{\mathrm{R}}\in \vecb{V}_h^\mathrm{R}$ can be split into
\begin{align*}
\vecb{v}_h^{\mathrm{R}}=\vecb{v}_h^{\mathrm{RT}_0}+\widetilde{\vecb{v}}_h^{\mathrm{R}}=\sum_{F\in\mathcal{F}^0}\mathrm{dof}_F(\vecb{v}_h^{\mathrm{RT}_0})\vecb{\psi}_F+\widetilde{\vecb{v}}_h^{\mathrm{R}}\in \vecb{RT}_0(\mathcal{T})\oplus \widetilde{\vecb{RT}}_{k-1}^{\mathrm{int}}(\mathcal{T}),
\end{align*}
where $\vecb{\psi}_F\in \vecb{RT}_0(\mathcal{T})$ is the standard face basis corresponding to $F\in\mathcal{F}^0$ and the operators $\mathrm{dof}_F: \vecb{RT}_0(\mathcal{T})\rightarrow\mathbb{R}, F\in\mathcal{F}^0,$ satisfy
\begin{align*}
\mathrm{dof}_F(\vecb{v}_h^{\mathrm{RT}_0}):=\int_F \vecb{v}_h^{\mathrm{RT}_0}\cdot\vecb{n}_F \ds\Big/\int_F \vecb{\psi}_F\cdot \vecb{n}_F\ds\quad\text{for all } \vecb{v}_h^{\mathrm{RT}_0}\in \vecb{RT}_0(\mathcal{T}),
\end{align*}
with $\vecb{n}_F$ being an unit normal vector of $F$. For $k=1$, $\widetilde{\vecb{v}}_h^{\mathrm{R}}$ is always zero; for $k\geq d$, $\vecb{v}_h^{\mathrm{RT}_0}$ is always zero.

Define $a(\vecb{u},\vecb{v}):=(\nabla \vecb{u},\nabla\vecb{v})$ for all $\vecb{u},\vecb{v}\in\vecb{V}$. For \(\vecb{v}_h := (\vecb{v}_h^{\mathrm{ct}}, \vecb{v}_h^{\mathrm{R}}) \in \vecb{V}_h^{\mathrm{ct}} \times \vecb{V}_h^\mathrm{R}\) consider
the (non-symmetric) bilinear form
\begin{align*}
  a_h(\vecb{u}_h,\vecb{v}_h)
  := a(\vecb{u}_h^{\mathrm{ct}}, \vecb{v}_h^{\mathrm{ct}})
     - (\Delta_\text{pw} \vecb{u}_h^{\mathrm{ct}}, \vecb{v}_h^{\mathrm{R}})
     + (\Delta_\text{pw} \vecb{v}_h^{\mathrm{ct}}, \vecb{u}_h^{\mathrm{R}})
     + a_h^D(\vecb{u}_h^{\mathrm{RT}_0}, \vecb{v}_h^{\mathrm{RT}_0}),
\end{align*}
where \(\Delta_\text{pw}\) is the piecewise Laplacian operator with respect to \(\mathcal{T}\). Note that the Laplacian terms vanish for $k=1$, which is the case in \cite{li2021low}.
The stabilization $a_h^D: \vecb{RT}_0(\mathcal{T})\times \vecb{RT}_0(\mathcal{T})\rightarrow \mathbb{R}$ is the same as \cite{li2021low} and has several equivalent choices which satisfy
\begin{align}\label{eqn:ahD_equivalence}
  a_h^D(\vecb{v}_h^{\mathrm{RT}_0}, \vecb{v}_h^{\mathrm{RT}_0}) \approx \|\alpha^{1/2} h^{-1}_\mathcal{T} \vecb{v}_h^{\mathrm{RT}_0} \|^2,
\end{align}
where $\alpha$ is a given positive piecewise constant.
For simplicity, $\alpha$ is chosen to be constant
on the whole domain $\Omega$.
In numerical experiments, we choose the form which results in a diagonal block
\begin{align*}
  a_h^D(\vecb{u}_h^{\mathrm{RT}_0}, \vecb{v}_h^{\mathrm{RT}_0})
  := \alpha \sum_{F\in\mathcal{F}^0}\mathrm{dof}_F(\vecb{u}_h^{\mathrm{RT}_0}) \mathrm{dof}_F(\vecb{v}_h^{\mathrm{RT}_0}) \, (\mathrm{div} \vecb{\psi}_F, \mathrm{div} \vecb{\psi}_F).
\end{align*}
Note, that the method is stabilization-free for $k\geq d$.
\begin{remark}
As in DG methods for elliptic equations (cf. \cite{ABCM2002}) or for the viscosity term in
the Stokes equations (cf. \cite[Sec. 4.4]{JLMNR:2017}), the second term and the fourth term of
$a_h$ are added to guarantee consistency and coercivity, respectively.
The choice of the third term can be different. Here, a term which is
skew-symmetric to the second term is used. However, one can also employ a symmetric one
(i.e., $-(\Delta_\text{pw} \vecb{v}_h^{\mathrm{ct}}, \vecb{u}_h^{\mathrm{R}})$). The analysis of the latter case is
indeed very similar to the skew-symmetric case except
that one should require that $\alpha$ is sufficiently large to guarantee coercivity and a
similar stabilization should also be added to the
$\widetilde{\vecb{RT}}_{k-1}^{\mathrm{int}}(\mathcal{T})$ part then. The reason for choosing the
skew-symmetric form here is that our method is parameter-free in this case for $k\geq d$ due to
the divergence constraint. The numerical experiments show that this non-symmetry does
not affect the convergence rate of the $L^2$ norm.
\end{remark}

On the product space, the bilinear form for the divergence constraint reads
\begin{align*}
  b(\vecb{v}_h,q_h) := -(\mathrm{div}(\vecb{v}_h^{\mathrm{ct}}+\vecb{v}_h^{\mathrm{R}}), q_h).
\end{align*}
The full discrete problem seeks \((\vecb{u}_h,p_h) \in \vecb{V}_{h}\times Q_h\), such that
\begin{equation}\label{eq:fullscheme}
\begin{aligned}
  \nu a_h(\vecb{u}_h, \vecb{v}_h) + b(\vecb{v}_h,p_h)
  & = (\vecb{f}, \vecb{v}_h)
& \text{for all } \vecb{v}_h \in \vecb{V}_{h},\\
  b(\vecb{u}_h, q_h)
  & = 0
& \text{for all } q_h \in Q_{h}.
\end{aligned}
\end{equation}
Here and throughout the rest of the paper, the $\vecb{L}^2$ inner product on the right-hand side
is to be understood as
\begin{align*}
(\vecb{f}, \vecb{v}_h) := (\vecb{f}, \vecb{v}_h^\mathrm{ct}+\vecb{v}_h^\mathrm{R}).
\end{align*}
In the space of discretely divergence-free functions
\begin{align*}
  \vecb{V}_{h,0} := &\left\lbrace \vecb{v}_h = (\vecb{v}_h^{\mathrm{ct}}, \vecb{v}_h^{\mathrm{R}}) \in \vecb{V}_h : b(\vecb{v}_h,q_h) = 0
  \text { for all } q_h\in Q_h \right\rbrace\\
  = & \left\lbrace \vecb{v}_h = (\vecb{v}_h^{\mathrm{ct}}, \vecb{v}_h^{\mathrm{R}}) \in \vecb{V}_h : \mathrm{div}(\vecb{v}_h^{\mathrm{ct}} + \vecb{v}_h^{\mathrm{R}}) = 0
   \right\rbrace,
\end{align*}
the above problem is also equivalent to
\begin{equation}\label{eq:reduced_scheme}
  \nu a_h(\vecb{u}_h, \vecb{v}_h)
   = (\vecb{f}, \vecb{v}_h)
\quad \text{for all } \vecb{v}_h \in \vecb{V}_{h,0}.
\end{equation}
Note, that \(a_h\) can be naturally extended to \(a_h : (\vecb{V} \times \vecb{V}_h^\mathrm{R}) \times (\vecb{V} \times \vecb{V}_h^\mathrm{R}) \rightarrow \mathbb{R}\) and similarly for \(b\).
We introduce two seminorms $||| \bullet |||$ and $|||\bullet |||_\star$ on $\vecb{V} \times \vecb{V}_h^\mathrm{R}$ which are characterized by
\begin{equation}\label{eq:norm}
   ||| \vecb{v} |||^2 :=a_h(\vecb{v}, \vecb{v})\quad\text{and}\quad
   ||| \vecb{v} |||^2_\star := ||| \vecb{v} |||^2 +\|h_\mathcal{T}\Delta_\text{pw} \vecb{v}^{\mathrm{ct}}\|^2+\|\mathrm{div}(\widetilde{\vecb{v}}^\mathrm{R})\|^2,
\end{equation}
for all $\vecb{v}=:(\vecb{v}^\mathrm{ct}, \vecb{v}^\mathrm{R})=:(\vecb{v}^\mathrm{ct}, \vecb{v}^{\mathrm{RT}_0}+\widetilde{\vecb{v}}^\mathrm{R})\in \vecb{V} \times \vecb{V}_h^\mathrm{R}$, respectively.
\begin{lemma}
  \label{lem:norm_uniqueness}
  $||| \bullet |||$ is a norm on $\vecb{V}_{h,0}$.
\end{lemma}
\begin{proof}
It suffices to prove that
\begin{displaymath}
a_{h}(\vecb{v}_{h},\vecb{v}_{h})=0 \Rightarrow \vecb{v}_h^{\mathrm{ct}}=\vecb{v}_h^{\mathrm{R}}=\vecb{v}_{h}^{\mathrm{RT}_0}+\widetilde{\vecb{v}}_{h}^{\mathrm{R}}=\vecb{0}.
\end{displaymath}
Consider $\vecb{v}_{h} \in \vecb{V}_{h,0}$ with $0 = a_{h}(\vecb{v}_{h},\vecb{v}_{h})=\|\nabla\vecb{v}_h^{\mathrm{ct}}\|^{2}+\| \vecb{v}_h^{\mathrm{RT}_0}\|^2_D$, where $\|\bullet\|_D$ is a natural norm on $\vecb{RT}_0(\mathcal{T})\cap \vecb{H}_0(\mathrm{div},\Omega)$ from $a_h^D(\bullet,\bullet)$.

First, $\| \nabla \bullet \|$ and $\| \bullet \|_D$ being norms implies
$\vecb{v}_h^{\mathrm{ct}}=\vecb{v}_{h}^{\mathrm{RT}_0}=\vecb{0}$.

Second, since also $\mathrm{div}(\vecb{v}_{h})=\mathrm{div}(\vecb{v}_h^{\mathrm{ct}})+\mathrm{div}(\vecb{v}_h^{\mathrm{RT}_0})+\mathrm{div}(\widetilde{\vecb{v}}_{h}^{\mathrm{R}})=0$,
we have $\mathrm{div}(\widetilde{\vecb{v}}_{h}^{\mathrm{R}})=0$ and by injectivity of \(\mathrm{div}|_{\widetilde{\vecb{RT}}_{k-1}^{\mathrm{int}}(\mathcal{T})}\) it follows $\widetilde{\vecb{v}}_{h}^{\mathrm{R}}=\vecb{0}$ also. This completes the proof.
\end{proof}

\begin{lemma}
  \label{lem:norm}
  $||| \bullet |||_\star$ is a norm on $\vecb{V} \times \vecb{V}_h^\mathrm{R}$.
\end{lemma}
\begin{proof}
The additional term $\| \mathrm{div}(\widetilde{\vecb{v}}^{\mathrm{R}}) \|^2$ implies $\mathrm{div}(\widetilde{\vecb{v}}^{\mathrm{R}})=0$ also for any $\vecb{v}\in \vecb{V} \times \vecb{V}_h^\mathrm{R}$ if $||| \vecb{v} |||_\star=0$. Then Lemma~\ref{lem:norm} follows with a similar analysis as in Lemma~\ref{lem:norm_uniqueness}.
\end{proof}

Note that the two norms are equivalent on $\vecb{V}_{h,0}$ due to an inverse inequality and $\mathrm{div}(\widetilde{\vecb{v}}_h^\mathrm{R})=-\mathrm{div}(\vecb{v}_h^\mathrm{ct})-\mathrm{div}(\vecb{v}_h^\mathrm{RT_0})$ for $\vecb{v}_h\in \vecb{V}_{h,0}$. The Laplacian term in $|||\bullet|||_\star$ is mainly introduced to prove the boundedness of $a_h$ on $(\vecb{V} \times \vecb{V}_h^\mathrm{R}) \times (\vecb{V} \times \vecb{V}_h^\mathrm{R})$.

\begin{remark}[Remark on uniqueness]
It has to be stressed that, apart from the lowest-order case as shown in \cite{li2021low}, the spaces \(\vecb{V}_h^{\mathrm{ct}}\) and \(\vecb{V}_h^\mathrm{R}\)
might have a non-zero intersection.
That means there might be some
functions \(\vecb{w} \in \vecb{V}^{\mathrm{ct}}_h \cap\vecb{V}^{\mathrm{R}}_h\) which can be represented by
\((\beta\vecb{w},(1-\beta)\vecb{w})\) for arbitrary $\beta\in\mathbb{R}$ in the product space \(\vecb{V}_h^{\mathrm{ct}} \times \vecb{V}_h^\mathrm{R}\).
Such an example can be found, e.g.,\ for $k=2$ on a partition as in Fig. \ref{figure1}. It is not difficult to verify that $\varphi_{3}\vecb{\psi}^{\mathrm{RT}_0}_3\in \vecb{V}_h^\mathrm{R}$ is also in $\vecb{V}_h^{\mathrm{ct}}$, because $\vecb{\psi}^{\mathrm{RT}_0}_3=2(\vecb{x}-\vecb{m})$ on the whole domain with $\vecb{m}=(0.5,0.5)^{\top}$ (which is certainly continuous) and $\varphi_3$ is continuous with vanishing boundary value.
\begin{figure}
\centering
\includegraphics[width=0.25\textwidth, height=0.25\textwidth]{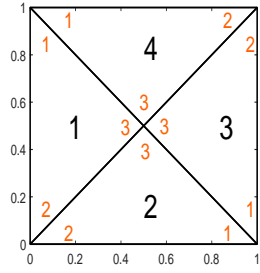}
\caption{A mesh of $(0,1)\times(0,1)$. Orange numbers denote the local series number of vertices, and black numbers denote the element series number.}
\label{figure1}
\end{figure}
However, if a solution exists (which is proven also later) our method selects a unique representation in $\vecb{V}_{h,0}$ according
to Lemma~\ref{lem:norm_uniqueness}.
\end{remark}

\section{Inf-sup stability}
\label{sec:infsup}
This section proves the inf-sup stability of the proposed method for which we assume inf-sup stability of $\vecb{V}^\mathrm{ct}_{h} \times \widehat{Q}_h$ for some auxiliary pressure space
\(\widehat{Q}_h\) as explained in Subsection~\ref{sec:enrichment_principle}.
We first introduce two operators $\mathcal{R}: \vecb{V}\rightarrow \vecb{V}_h^\mathrm{R}\cap\widetilde{\vecb{RT}}_{k-1}^{\mathrm{int}}(\mathcal{T})$ and
$\widetilde{\mathcal{R}}: Q\rightarrow \vecb{V}_h^\mathrm{R}\cap\widetilde{\vecb{RT}}_{k-1}^{\mathrm{int}}(\mathcal{T})$
which are characterized by
\begin{align*}
\mathrm{div} (\widetilde{\mathcal{R}}q):=\pi_{\widehat{Q}_h^\perp\cap \widetilde{P}_{k-1}^{\mathrm{disc}}(\mathcal{T})} q=\begin{cases}\pi_{\widehat{Q}_h^\perp} q &k\geq d,\\
\pi_{\widetilde{P}_{k-1}^{\mathrm{disc}}(\mathcal{T})} q \quad &k<d,
\end{cases}
\quad\text{for all } q\in Q,
\end{align*}
with $\widetilde{P}_{k-1}^{\mathrm{disc}}(\mathcal{T})$ as defined in \eqref{eq:def_tildepk_disc} and
\begin{align*}
\mathcal{R}\vecb{v}:=\widetilde{\mathcal{R}}\left(\mathrm{div}\left(\vecb{v}\right)\right) \quad
\text{for all } \vecb{v}\in \vecb{V}.
\end{align*}
Due to the bijective property of the divergence operator with respect to the corresponding spaces, these two operators are well-defined. From the definition of $\mathcal{R}$ and $\widetilde{\mathcal{R}}$ one has
\begin{equation}\label{eq:operatorR_commute}
\mathrm{div}\left(\mathcal{R}\vecb{v}\right)=\pi_{\widehat{Q}_h^\perp\cap \widetilde{P}_{k-1}^{\mathrm{disc}}(\mathcal{T})} \mathrm{div}\left(\vecb{v}\right).
\end{equation}
Meanwhile, according to the stability of $L^2$ projection, we have
\begin{equation}\label{ieq:L2stability}
\|\mathrm{div}(\widetilde{\mathcal{R}}q)\|\leq \|q\| \quad\text{and} \quad \left\|\mathrm{div}\left(\mathcal{R}\vecb{v}\right)\right\|\leq \left\|\mathrm{div}(\vecb{v})\right\|\leq \left\|\nabla\vecb{v}\right\| \quad\text{for all } q\in Q, \vecb{v}\in \vecb{V}.
\end{equation}

\subsection{Fortin operator}
  Due to the inf-sup stability of \((\vecb{V}_h^{\mathrm{ct}},\widehat{Q}_h)\), there exists an $\vecb{H}^1$-stable Fortin operator $\Pi^{\mathrm{ct}} : \vecb{V} \rightarrow {\vecb{V}}_{h}^{\mathrm{ct}}$, i.e.,
  \begin{equation}\label{ieq:H1stability_ct}
  \|\nabla\Pi^{\mathrm{ct}}\vecb{v}\|\leq C_{\mathrm{F}}^{\mathrm{ct}} \|\nabla\vecb{v}\|,
  \end{equation}
  and
  \begin{equation}\label{eq:Fortin_ct}
     (\mathrm{div}(\Pi^{\mathrm{ct}} \vecb{v}), q_h)=(\mathrm{div}(\vecb{v}), q_h) \quad \text{for all } q_h \in \widehat{Q}_h.
  \end{equation}
  In cases where \(\widehat{Q}_h = \lbrace 0 \rbrace\), a quasi-interpolation operator \(\Pi^{\mathrm{ct}}\), e.g.\ the one from \cite[Section 4.8]{BrennerScott:2008},
  is chosen, which satisfies
  \begin{equation}\label{ieq:appro_pict}
   \quad \|\vecb{v}-\Pi^{\mathrm{ct}} \vecb{v}\|_{T} \lesssim h_{T}\left\|\nabla\vecb{v}\right\|_{\omega(T)}\quad\text{for all } T\in \mathcal{T},\vecb{v}\in\vecb{V},
  \end{equation}
  with $\omega(T)$ being the union of all nodal patches of vertices in $T$ \cite{LLMS2017}.
  \begin{lemma}\label{lem:Fortin_operator}
   The operator \(\Pi : \vecb{V} \rightarrow \vecb{V}_h\) defined by
  \begin{align*}
    \Pi \vecb{v} := \begin{cases}\left(\Pi^{\mathrm{ct}} \vecb{v}, \mathcal{R} \widehat{\vecb{v}}\right)\quad &k\geq d,\\
    \left(\Pi^{\mathrm{ct}} \vecb{v},\Pi^{\mathrm{RT_0}}\widehat{\vecb{v}}+ \mathcal{R} \widehat{\vecb{v}}\right)\quad &k<d,
    \end{cases}\quad\text{for all } \vecb{v}\in \vecb{V},
  \end{align*}
  is a Fortin operator, i.e.,
  \begin{align*}
  |||\Pi\vecb{v}|||_\star\leq C_F\|\nabla\vecb{v}\|\quad \text{and} \quad b(\Pi\vecb{v},q_h)=-(\mathrm{div}(\vecb{v}),q_h) \quad \text{for all } \vecb{v}\in \vecb{V}, q_h\in Q_h,
  \end{align*}
  where $\widehat{\vecb{v}}:=(1 - \Pi^{\mathrm{ct}})\vecb{v}$ and $\Pi^{\mathrm{RT_0}}$ is the standard lowest-order Raviart--Thomas interpolation.
  \end{lemma}
  \begin{proof}
  Recall the approximation property of $\Pi^{\mathrm{RT_0}}$ \cite{BBF:book:2013},
  \begin{equation}\label{ieq:appro_pirt0}
   \|\vecb{v}-\Pi^{\mathrm{RT}_0} \vecb{v}\|_{T} \lesssim h_{T}\left\|\nabla\vecb{v}\right\|_{T} \quad\text{for all } T\in \mathcal{T},\vecb{v}\in\vecb{V}.
  \end{equation}
  Note that \eqref{ieq:L2stability} and \eqref{ieq:H1stability_ct} yield
  \begin{align*}
\|\mathrm{div}(\mathcal{R}\widehat{\vecb{v}})\|\leq \|\mathrm{div}(\widehat{\vecb{v}})\|\leq \|\mathrm{div}(\vecb{v})\|+
  \|\mathrm{div}(\Pi^{\mathrm{ct}}\vecb{v})\|\leq (1+C_\mathrm{F}^\mathrm{ct})\|\nabla\vecb{v}\|.
  \end{align*}
 For $k<d$, \eqref{ieq:appro_pict} and \eqref{ieq:appro_pirt0} show
  \begin{align*}
  h_T^{-1}\|\Pi^{\mathrm{RT_0}}\widehat{\vecb{v}}\|_T\lesssim h_T^{-1}\left\|\widehat{\vecb{v}}\right\|_T+\left\|\nabla\widehat{\vecb{v}}\right\|_T\lesssim \left\|\nabla\vecb{v}\right\|_{\omega(T)}+\|\nabla\Pi^{\mathrm{ct}}\vecb{v}\|_T.
  \end{align*}
  Then the stability property
  \begin{align*}
  |||\Pi \vecb{v}|||_\star\lesssim \left(|||\Pi\vecb{v}|||^2+\left\|\mathrm{div}\left(\mathcal{R}\widehat{\vecb{v}}\right)\right\|^2\right)^{1/2}\lesssim \left\|\nabla\vecb{v}\right\|
  \end{align*}
  follows from an inverse inequality, \eqref{ieq:H1stability_ct} and the above two inequalities.

  Let us prove $b(\Pi\vecb{v},q_h)=-(\mathrm{div}(\vecb{v}),q_h)$ for all $q_h\in Q_h$.
  For $k\geq d$ and for any $q_h\in Q_h$,
  due to $\mathrm{div}(\mathcal{R}\widehat{\vecb{v}}) = \pi_{\widehat{Q}_h^\perp}\mathrm{div}(\widehat{\vecb{v}})$ by \eqref{eq:operatorR_commute} it holds
  \begin{align*}
  \left(\mathrm{div}\left(\mathcal{R} \widehat{\vecb{v}}\right),q_h\right)
  =\left(\pi_{\widehat{Q}_h^\perp}\mathrm{div}\left(\widehat{\vecb{v}}\right),\pi_{\widehat{Q}_h^\perp}q_h\right)
  =\left(\mathrm{div}\left(\vecb{v}-\Pi^{\mathrm{ct}}\vecb{v}\right),\pi_{\widehat{Q}_h^\perp}q_h\right),
  \end{align*}
  which implies that
  \begin{align*}
  b\left(\Pi\vecb{v},q_h\right)
  &=-\left(\mathrm{div}\left(\Pi^{\mathrm{ct}}\vecb{v}\right),\pi_{\widehat{Q}_h}q_h\right)
  -\left(\mathrm{div}\left(\Pi^{\mathrm{ct}}\vecb{v}\right),\pi_{\widehat{Q}_h^\perp}q_h\right)\\
  &\hspace{1cm}-\left(\mathrm{div}\left(\vecb{v}-\Pi^{\mathrm{ct}}\vecb{v}\right),\pi_{\widehat{Q}_h^\perp}q_h\right)\\
  (\text{by } \eqref{eq:Fortin_ct})\quad&=-\left(\mathrm{div}\left(\vecb{v}\right),\pi_{\widehat{Q}_h}q_h\right)-
  \left(\mathrm{div}\left(\vecb{v}\right),\pi_{\widehat{Q}_h^\perp}q_h\right)=-\left(\mathrm{div}\left(\vecb{v}\right),q_h\right).
  \end{align*}
  For $k<d$, one obtains similarly
  \begin{align*}
  \left(\mathrm{div}\left(\Pi^{\mathrm{RT}_0} \widehat{\vecb{v}}\right),q_h\right)
  =\left(\mathrm{div}\left(\vecb{v}-\Pi^{\mathrm{ct}}\vecb{v}\right),\pi_{{P}_{0}^{\mathrm{disc}}(\mathcal{T})}q_h\right),
  \end{align*}
  and
  \begin{align*}
  \left(\mathrm{div}\left(\mathcal{R} \widehat{\vecb{v}}\right),q_h\right)
  =\left(\mathrm{div}\left(\vecb{v}-\Pi^{\mathrm{ct}}\vecb{v}\right),\pi_{\widetilde{P}_{k-1}^{\mathrm{disc}}(\mathcal{T})}q_h\right).
  \end{align*}
  Then the $L^2$-orthogonality between ${P}_{0}^{\mathrm{disc}}$ and $\widetilde{P}_{k-1}^{\mathrm{disc}}$ gives
  \begin{align*}
  \left(\mathrm{div}\left(\mathcal{R} \widehat{\vecb{v}}\right),q_h\right)+\left(\mathrm{div}\left(\Pi^{\mathrm{RT}_0} \widehat{\vecb{v}}\right),q_h\right)
  =\left(\mathrm{div}\left(\vecb{v}-\Pi^{\mathrm{ct}}\vecb{v}\right),q_h\right).
  \end{align*}
  It follows from the above equality that
  \begin{align*}
  b\left(\Pi\vecb{v},q_h\right)
  =-\left(\mathrm{div}\left(\vecb{v}\right),q_h\right).
  \end{align*}
  This completes the proof.
  \end{proof}

  The existence of a Fortin interpolator implies that the discrete inf-sup condition holds, which
  is stated in the following theorem.
\begin{theorem}[Inf-sup stability]
  There exists a constant \(\beta > 0\) independent of $h$
  such that
  \begin{align*}
    \sup_{\vecb{v}_h \in \vecb{V}_h} \frac{b(\vecb{v}_h,q_h)}{||| \vecb{v}_h |||_\star} \geq \beta \|q_h\|
    \quad \text{for all } q_h \in Q_h.
  \end{align*}
  \end{theorem}
  
Together with \(a_h\) being coercive and bounded on $\vecb{V}_{h,0}$ (see Section~\ref{sec:apriori_estimates} below), the discrete problem is well-posed and uniquely solvable.
In summary, the Raviart--Thomas enrichment of the Scott--Vogelius pairs leads to inf-sup stable pairs on general shape-regular simplicial grids.

\section{A priori error analysis}
\label{sec:apriori_estimates}

For the subsequent analysis recall that the bilinear form $a_h$ is extended to $\vecb{V} \times \vecb{V}_h$ and the solution $\vecb{u}$ of \eqref{eq:stokes_weak} satisfies
\begin{align*}
  a_h((\vecb{u},\vecb{0}), \vecb{v}_h) = a(\vecb{u}, \vecb{v}^{\mathrm{ct}}_h) - (\Delta \vecb{u}, \vecb{v}_h^{\mathrm{R}})
  \quad \text{for any } \vecb{v}_h \in \vecb{V}_h.
\end{align*}

\begin{lemma}[Consistency]
  \label{lem:consistency}
For the solution \(\vecb{u} \in \vecb{V}_0\) of \eqref{eq:stokes_weak} with \(\Delta \vecb{u} \in \vecb{L}^2(\Omega)\), it holds
\begin{align*}
  \nu a_h((\vecb{u},\vecb{0}), \vecb{v}_h) = (\vecb{f},\vecb{v}_h) \quad \text{for all } \vecb{v}_h \in \vecb{V}_{h,0}.
\end{align*}
Moreover, if \(\vecb{u} \in \vecb{V}_h^{\mathrm{ct}}\) then it is \(\vecb{u}_h = (\vecb{u},\vecb{0})\).
\end{lemma}
\begin{proof}
  The first statement follows from
   \begin{align*}
   \nu a_h((\vecb{u},\vecb{0}), \vecb{v}_h)
   = \nu (\nabla \vecb{u}, \nabla \vecb{v}_h^{\mathrm{ct}})
     - \nu (\Delta \vecb{u}, \vecb{v}_h^{\mathrm{R}})
   =  - \nu (\Delta \vecb{u}, \vecb{v}_h^{\mathrm{ct}} + \vecb{v}_h^{\mathrm{R}})
   = (\vecb{f}, \vecb{v}_h^{\mathrm{ct}} + \vecb{v}_h^{\mathrm{R}}).
   \end{align*}
   If \(\vecb{u} \in \vecb{V}_h^{\mathrm{ct}}\), this also shows that \(\vecb{u}_h = (\vecb{u},\vecb{0})\) is
   a discrete solution and due to unique solvability also the only one.
\end{proof}
\begin{lemma}\label{lem:bound_vhr}
    The following inequality is valid for $\vecb{v}_h \in \widetilde{\vecb{RT}}_k^{\mathrm{int}}(\mathcal{T})$:
    \begin{equation}\label{ieq:L2_L2div}
    \left\|\vecb{v}_h\right\|_{T}\leq C h_{T} \left\|\mathrm{div}\left(\vecb{v}_h\right)\right\|_{T}\quad \text{for all } T\in \mathcal{T}.
    \end{equation}
  \end{lemma}
  \begin{proof}
  Recall that we have required in Subsection~\ref{sec:notation} that all local spaces $\widetilde{\vecb{RT}}_k^{\mathrm{int}}(T)$ are obtained by Piola's transformation from a same reference space, which is denoted by $\widetilde{\vecb{RT}}_k^{\mathrm{int}}(T^\mathrm{ref})$ in what follows.
  Due to the injectivity of the divergence operator, both $\|\bullet\|_{T^\mathrm{ref}}$ and $\|\mathrm{div}\bullet\|_{T^\mathrm{ref}}$ are norms on $\widetilde{\vecb{RT}}_k^{\mathrm{int}}(T^\mathrm{ref})$. Since the dimension of $\widetilde{\vecb{RT}}_k^{\mathrm{int}}(T^\mathrm{ref})$ is finite, the two norms are equivalent, i.e., there must be two constants
  $C_\star$ and $C^\star$ which satisfy
  \begin{equation}\label{ieq:ref_fun}
  C_\star\left\|\mathrm{div}\left(\widehat{\vecb{v}}_h\right)\right\|_{T^\mathrm{ref}}\leq\left\|\widehat{\vecb{v}}_h\right\|_{T^\mathrm{ref}}\leq C^\star \left\|\mathrm{div}\left(\widehat{\vecb{v}}_h\right)\right\|_{T^\mathrm{ref}}\quad\text{for all }\widehat{\vecb{v}}_h\in \widetilde{\vecb{RT}}_k^{\mathrm{int}}(T^\mathrm{ref}).
  \end{equation}
  The constants $C_\star$ and $C^\star$ are independent of $h$ since the above inequality is related to reference space only.

  For any $\vecb{v}_h\in \widetilde{\vecb{RT}}_k^{\mathrm{int}}(T)$, there exists a corresponding reference function $\widehat{\vecb{v}}_h\in\widetilde{\vecb{RT}}_k^{\mathrm{int}}(T^\mathrm{ref})$. Then a combination of \eqref{ieq:ref_fun} and \cite[Eqs. 2.1.75 \& 2.1.78]{BBF:book:2013} gives \eqref{ieq:L2_L2div}.
  \end{proof}
\begin{lemma}[Coercivity and boundedness of $a_h$]\label{lem:coercivity}
The bilinear form $a_h$ satisfies
\begin{align*}
a_h(\vecb{v}_h,\vecb{v}_h)=|||\vecb{v}_h|||^2\gtrsim |||\vecb{v}_h|||_\star^2 \quad \text{for all } \vecb{v}_h\in \vecb{V}_{h,0},
\end{align*}
and
\begin{align*}
a_h(\vecb{u},\vecb{v})\lesssim |||\vecb{u}|||_\star|||\vecb{v}|||_\star \quad \text{for all } \vecb{u}, \vecb{v}\in \vecb{V} \times \vecb{V}_h^\mathrm{R}.
\end{align*}
\end{lemma}
\begin{proof}
The coercivity of $a_h$ is obvious. Let us prove the boundedness. It follows from the Cauchy--Schwarz inequality and Lemma~\ref{lem:bound_vhr} that
\begin{align*}
a_h(\vecb{u},\vecb{v})&=(\nabla \vecb{u}^{\mathrm{ct}},\nabla \vecb{v}^{\mathrm{ct}})-(\Delta_\text{pw} \vecb{u}^{\mathrm{ct}},\vecb{v}^{\mathrm{R}})
+(\Delta_\text{pw} \vecb{v}^{\mathrm{ct}},\vecb{u}^{\mathrm{R}})+a_h^D(\vecb{u}^{\mathrm{RT}_0},\vecb{v}^{\mathrm{RT}_0})\\
& \leq \|\nabla \vecb{u}^{\mathrm{ct}}\|\|\nabla \vecb{v}^{\mathrm{ct}}\|+\|h_{\mathcal{T}}\Delta_\text{pw} \vecb{u}^{\mathrm{ct}}\|\|h_{\mathcal{T}}^{-1}\vecb{v}^{\mathrm{R}}\|
+\|h_{\mathcal{T}}\Delta_\text{pw} \vecb{v}^{\mathrm{ct}}\|\|h_{\mathcal{T}}^{-1}\vecb{u}^{\mathrm{R}}\|\\
& \hspace{1cm} +a_h^D(\vecb{u}^{\mathrm{RT}_0},\vecb{u}^{\mathrm{RT}_0})^{1/2}a_h^D(\vecb{v}^{\mathrm{RT}_0},\vecb{v}^{\mathrm{RT}_0})^{1/2}\\
&\lesssim |||\vecb{u}|||_\star|||\vecb{v}|||_\star.
\end{align*}
This concludes the proof.
\end{proof}

\begin{theorem}[Pressure-robust a priori velocity error estimate]
  \label{thm:velocity_estimate}
Denote by \(\vecb{u} \in \vecb{V}_0\) and \(\vecb{u}_h \in \vecb{V}_{h,0}\) the velocity solutions of \eqref{eq:stokes_weak} and \eqref{eq:fullscheme}, respectively. Assume \(\vecb{u} \in \vecb{H}^{k+1}(\Omega)\). It holds
\begin{align*}
  ||| (\vecb{u},\vecb{0}) - \vecb{u}_h |||_\star & \lesssim \inf_{\vecb{v}_h \in \vecb{V}_{h,0}} ||| (\vecb{u},\vecb{0}) - \vecb{v}_h |||_\star
  \leq (1 + C_F) \inf_{\vecb{v}_h^\mathrm{ct} \in \vecb{V}_{h}^\mathrm{ct}} \|\nabla( \vecb{u} - \vecb{v}_h^\mathrm{ct})\| \lesssim h^k \lvert \vecb{u}\rvert_{\vecb{H}^{k+1}(\Omega)}.
\end{align*}
Here, \(C_F\) denotes the stability constant of the Fortin interpolator $\Pi$.
The above inequality implies
\begin{align*}
  || \nabla( \vecb{u} - \vecb{u}^{\mathrm{ct}}_h) ||^2
  + \| h_\mathcal{T}^{-1} \vecb{u}_h^{\mathrm{R}} \|^2
  \lesssim ||| (\vecb{u},\vecb{0}) - \vecb{u}_h |||^2_\star \leq h^{2k} \lvert \vecb{u}\rvert_{\vecb{H}^{k+1}(\Omega)}^2.
\end{align*}

\end{theorem}
\begin{proof}
 Lemma~\ref{lem:consistency} yields that
  \begin{align*}
    a_h((\vecb{u},\vecb{0}) - \vecb{u}_h, \vecb{v}_h) = 0 \quad \text{for all } \vecb{v}_h \in \vecb{V}_{h,0}.
  \end{align*}
  Therefore, it follows from Lemma~\ref{lem:coercivity} that
  \begin{align*}
  |||\vecb{u}_h - \vecb{v}_h|||_\star^2&\lesssim a_h(\vecb{u}_h - \vecb{v}_h, \vecb{u}_h - \vecb{v}_h)=a_h((\vecb{u},\vecb{0}) - \vecb{v}_h, \vecb{u}_h - \vecb{v}_h)\\
  &\lesssim |||(\vecb{u},\vecb{0}) - \vecb{v}_h|||_\star|||\vecb{u}_h - \vecb{v}_h|||_\star \quad \text{for any } \vecb{v}_h \in \vecb{V}_{h,0}.
  \end{align*}
  This and a triangle inequality shows
\begin{align*}
  || \nabla( \vecb{u} - \vecb{u}^{\mathrm{ct}}_h) ||
  \leq ||| (\vecb{u},\vecb{0}) - \vecb{u}_h |||_\star
  \leq ||| (\vecb{u},\vecb{0}) - \vecb{v}_h |||_\star+
   ||| \vecb{u}_h-\vecb{v}_h |||_\star
   \lesssim ||| (\vecb{u},\vecb{0}) - \vecb{v}_h |||_\star.
\end{align*}
Since $\vecb{v}_h$ is arbitrary, one gets
\begin{align*}
||| (\vecb{u},\vecb{0}) - \vecb{u}_h |||_\star\lesssim\inf_{\vecb{v}_h \in \vecb{V}_{h,0}} ||| (\vecb{u},\vecb{0}) - \vecb{v}_h |||_\star.
\end{align*}
Consider now \(\vecb{w}_h=(\vecb{w}_h^\mathrm{ct},0)\) with any $\vecb{w}_h^\mathrm{ct}$ in \(\vecb{V}_h^{\mathrm{ct}}\) and choose
\(\vecb{v}_h := \Pi(\vecb{u} - \vecb{w}_h^\mathrm{ct}) + \vecb{w}_h \in \vecb{V}_{h,0}\).
Due to the properties of the Fortin operator, one obtains
\begin{align*}
  ||| (\vecb{u},\vecb{0}) - \vecb{v}_h |||_\star
  \leq ||| (\vecb{u},\vecb{0})- \vecb{w}_h |||_\star + ||| \Pi(\vecb{u} - \vecb{w}_h^\mathrm{ct}) |||_\star
  \leq (1 + C_F) \| \nabla (\vecb{u} - \vecb{w}_h^\mathrm{ct} )\|.
\end{align*}
Since \(\vecb{w}_h^\mathrm{ct}\) is arbitrary, this inequality also holds for the infimums on both sides. Hence, one arrives at
\begin{align*}
||| (\vecb{u},\vecb{0}) - \vecb{u}_h |||_\star\lesssim (1 + C_F) \inf_{\vecb{w}_h^\mathrm{ct} \in \vecb{V}^\mathrm{ct}_{h}} \| \nabla (\vecb{u} - \vecb{w}_h^\mathrm{ct} )\|.
\end{align*}
This finishes the proof.
\end{proof}

\begin{theorem}[A priori pressure error estimate]
  \label{thm:pressure_estimate}
  Denote by \(p \in Q\) and \(p_h \in Q_h\) the pressure solutions of \eqref{eq:stokes_weak} and \eqref{eq:fullscheme}, respectively.
  Assume $\vecb{u}\in \vecb{H}^{k+1}(\Omega)$ and \(p \in H^{k}(\Omega)\). It holds
  \begin{align*}
    \| p - p_h \| \lesssim \inf_{q_h \in Q_h} \| p - q_h \|
    + \nu ||| (\vecb{u},\vecb{0}) - \vecb{u}_h |||_\star
    \lesssim h^k \lvert {p}\rvert_{H^{k}(\Omega)} + \nu h^k \lvert \vecb{u}\rvert_{\vecb{H}^{k+1}(\Omega)}.
  \end{align*}
  \end{theorem}
  \begin{proof}
    First, there is an identity for the $L^2$ best-approximation \(\pi_{Q_h} p\)
    within \(Q_h\), which gives
    \begin{align*}
      \| p - p_h \|^2 = \inf_{q_h \in Q_h} \| p - q_h \|^2  + \| \pi_{Q_h} p - p_h \|^2.
    \end{align*}
    To estimate the last term, the inf-sup stability guarantees the existence of some
    \(\vecb{v}_h \in \vecb{V}_h\) such that \(\mathrm{div} (\vecb{v}_h^\mathrm{ct}+\vecb{v}_h^\mathrm{R}) = \pi_{Q_h} p - p_h\)
    and \(||| \vecb{v}_h |||_\star \lesssim \| \pi_{Q_h} p - p_h \|\). This allows for the estimate
    \begin{align*}
      \| \pi_{Q_h} p - p_h \|^2
      & = -b( \vecb{v}_h,\pi_{Q_h} p - p_h)
       = -b(\vecb{v}_h,p - p_h)\\
      & = \nu a_h((\vecb{u},\vecb{0}) - \vecb{u}_h, \vecb{v}_h)\\
      &  \lesssim \nu ||| (\vecb{u},\vecb{0}) - \vecb{u}_h |||_\star  \, ||| \vecb{v}_h |||_\star.
    \end{align*}
    Inserting the bound for \(||| \vecb{v}_h |||_\star\) concludes the proof.
  \end{proof}

\section{Reduction to divergence-free schemes with only $P_0$ pressure}
\label{sec:reduced_scheme}
This section derives a reduced scheme that allows to remove the additional
Raviart--Thomas degrees of freedom and only requires a $P_0$ pressure, effectively
resulting into a $\vecb{P}_{k} \times P_{0}$-like system.
Similar reduced schemes can be found in \cite{Lehrenfeld2010,Veiga17,Wei_VEM2021}. The idea of our scheme is based on the important fact that the divergence operator from
$\widetilde{\vecb{RT}}_{k-1}^\mathrm{int}(\mathcal{T})$ to $\widetilde{P}_{k-1}^\mathrm{disc}(\mathcal{T})$ is bijective.

Firstly a general framework is given for $k\geq d$, where it is shown that our method is equivalent to a computable $\vecb{P}_k \times \widehat{Q}_h$-like system. Then taking $\widehat{Q}_h={P}_0^{\mathrm{disc}}(\mathcal{T})\cap Q$ results in a $\vecb{P}_{k} \times P_{0}$-like system for arbitrary $k\geq d$. The reduced scheme for $k<d$ is seperately discussed.
\subsection{General framework for $k\geq d$}
\label{sec:reduced_scheme_general}
Due to the divergence constraint, there are a series of equivalent reduced schemes which seek the velocity in a subspace of $\vecb{V}_h$. The scheme \eqref{eq:reduced_scheme} is such an example which is commonly used in theoretical analysis. However, \eqref{eq:reduced_scheme} is not connected to practical computations in general because the divergence-free basis functions are usually non-trivial to construct. This subsection discusses a computable reduced scheme ($\vecb{P}_{k} \times \widehat{Q}_h$-like), where the operators $\mathcal{R}$ and $\widetilde{\mathcal{R}}$ play an important role again.

The velocity ansatz space that incorporates the divergence constraint by $\widehat{Q}_h^\perp$ reads
\begin{equation}\label{eq:def_tildeVh}
\widehat{\vecb{V}}_{h}:=\bigg\{\vecb{v}_{h}=(\vecb{v}_h^{\mathrm{ct}},\vecb{v}_h^{\mathrm{R}})\in \vecb{V}_{h}: \int_{\Omega}\mathrm{div}(\vecb{v}_h^{\mathrm{ct}}+\vecb{v}_h^{\mathrm{R}})\lambda_{h}\dx=0
\quad\forall \lambda_{h}\in \widehat{Q}_h^\perp\bigg\}.
\end{equation}
Due to the inclusion relationship $\mathrm{div}(\vecb{V}_h^{\mathrm{ct}}+\vecb{V}_h^\mathrm{R})\subseteq Q_h$ and the $L^2$-orthogonal relationship between $\widehat{Q}_h$ and $\widehat{Q}_h^\perp$, $\widehat{\vecb{V}}_{h}$ can be also characterized as
\begin{displaymath}
\widehat{\vecb{V}}_{h}=\bigg\{\vecb{v}_{h}=(\vecb{v}_h^{\mathrm{ct}},\vecb{v}_h^{\mathrm{R}})\in \vecb{V}_{h}: \mathrm{div}(\vecb{v}_h^{\mathrm{ct}}+\vecb{v}_h^{\mathrm{R}})\in \widehat{Q}_h\bigg\}.
\end{displaymath}

Then our method is also equivalent to
seeking \(\vecb{u}_h \in \widehat{\vecb{V}}_{h}\) and \(\hat p_h \in \widehat{Q}_h\), such that
\begin{equation}\label{eq:galerkin_reducedform0}
\begin{aligned}
  \nu a_h(\vecb{u}_h, \vecb{v}_h) + b(\vecb{v}_h,\hat{p}_h)
  & = (\vecb{f}, \vecb{v}_h)
& \text{for all } \vecb{v}_h \in \widehat{\vecb{V}}_{h},\\
  b(\vecb{u}_h, q_h)
  & = 0
& \text{for all } q_h \in \widehat{Q}_h.
\end{aligned}
\end{equation}
\begin{lemma}\label{lem:reduced_hatVh_structure}
The following identity holds:
\begin{align*}
\widehat{\vecb{V}}_h=\left\{\vecb{v}_h=\left(\vecb{v}_h^{\mathrm{ct}},-\mathcal{R}\vecb{v}_h^{\mathrm{ct}}\right): \vecb{v}_h^{\mathrm{ct}}\in \vecb{V}_h^{\mathrm{ct}}\right\}.
\end{align*}
\end{lemma}
\begin{proof}
This proof is based on the $L^2$-orthogonality between $\widehat{Q}_h$ and $\widehat{Q}_h^\perp$. For any $\vecb{v}_h^{\mathrm{ct}}\in \vecb{V}_h^{\mathrm{ct}}$, one has $\mathrm{div}(\vecb{v}_h^{\mathrm{ct}})+\mathrm{div}(-\mathcal{R}\vecb{v}_h^{\mathrm{ct}})=\mathrm{div}(\vecb{v}_h^{\mathrm{ct}})-\pi_{\widehat{Q}_h^\perp}\mathrm{div}(\vecb{v}_h^{\mathrm{ct}})=
\pi_{\widehat{Q}_h}\mathrm{div}(\vecb{v}_h^{\mathrm{ct}})\in \widehat{Q}_h$,
which implies that
\begin{align*}
\widehat{\vecb{V}}_h\supseteq\left\{\vecb{v}_h=\left(\vecb{v}_h^{\mathrm{ct}},-\mathcal{R}\vecb{v}_h^{\mathrm{ct}}\right): \vecb{v}_h^{\mathrm{ct}}\in \vecb{V}_h^{\mathrm{ct}}\right\}.
\end{align*}
Conversely, for any $\vecb{v}_h=(\vecb{v}_h^{\mathrm{ct}},\vecb{v}_h^{\mathrm{R}})\in \vecb{V}_h$ which satisfies $\mathrm{div}(\vecb{v}_h^{\mathrm{ct}}+\vecb{v}_h^{\mathrm{R}})\in \widehat{Q}_h$, it follows from the orthogonality that
\begin{align*}
\left(\mathrm{div}(\vecb{v}_h^{\mathrm{ct}}+\vecb{v}_h^{\mathrm{R}}),q_h\right)=0 \quad \text{for all } q_h\in \widehat{Q}_h^\perp.
\end{align*}
The above equality means that $\mathrm{div}(\vecb{v}_h^{\mathrm{R}})=-\pi_{\widehat{Q}_h^\perp}\mathrm{div}(\vecb{v}_h^{\mathrm{ct}})$ and further $\vecb{v}_h^{\mathrm{R}}$ is exactly $-\mathcal{R}\vecb{v}_h^{\mathrm{ct}}$, which implies that
\begin{align*}
\widehat{\vecb{V}}_h\subseteq\left\{\vecb{v}_h=\left(\vecb{v}_h^{\mathrm{ct}},-\mathcal{R}\vecb{v}_h^{\mathrm{ct}}\right): \vecb{v}_h^{\mathrm{ct}}\in \vecb{V}_h^{\mathrm{ct}}\right\}.
\end{align*}
This completes the proof.
\end{proof}
According to Lemma~\ref{lem:reduced_hatVh_structure} one can rewrite \eqref{eq:galerkin_reducedform0} as
seeking \(\vecb{u}_h^{\mathrm{ct}} \in {\vecb{V}}_{h}^{\mathrm{ct}}\) such that
\begin{equation}\label{eq:galerkin_reducedform}
\begin{aligned}
  \nu a_h((\vecb{u}_h^{\mathrm{ct}},-\mathcal{R}\vecb{u}_h^{\mathrm{ct}}), (\vecb{v}_h^{\mathrm{ct}},-\mathcal{R}\vecb{v}_h^{\mathrm{ct}})) - (\mathrm{div}(\vecb{v}_h^{\mathrm{ct}}),\hat{p}_h)
  & = (\vecb{f}, \vecb{v}_h^{\mathrm{ct}}-\mathcal{R}\vecb{v}_h^{\mathrm{ct}})
&& \text{for all } \vecb{v}_h^{\mathrm{ct}}\in {\vecb{V}}_{h}^{\mathrm{ct}},\\
  (\mathrm{div}(\vecb{u}_h^{\mathrm{ct}}), q_h)
  & = 0
&& \text{for all } q_h \in \widehat{Q}_h,
\end{aligned}
\end{equation}
In this rewriting we apply also the fact that $(\mathrm{div}(\mathcal{R}\vecb{v}_h^{\mathrm{ct}}),q_{h})=0$ for all $(\vecb{v}_h^{\mathrm{ct}},q_{h})\in \vecb{V}_h^{\mathrm{ct}}\times\widehat{Q}_h$.

The scheme \eqref{eq:galerkin_reducedform} is the so-called $\vecb{P}_k \times \widehat{Q}_h$-like scheme herein. The implementation of \eqref{eq:galerkin_reducedform} relies on the simple implementation of $\mathcal{R}$. Note that $\vecb{V}_h^\mathrm{R}$ consists of some interior cell functions. The computation of $\mathcal{R}$ can be done on each element $T$, and clearly we have $\mathcal{R}\vecb{v}_h^{\mathrm{ct}}|_T=\vecb{0}$ if $\mathrm{div}(\vecb{v}_h^{\mathrm{ct}})|_T\in \widehat{Q}_h$. Denote by $\{\vecb{\psi}_{j},j=1,...,N_{0}\}$ the basis of $\vecb{V}_h^\mathrm{R}(T)$, where $N_{0}=\operatorname{dim}(\vecb{V}_h^\mathrm{R}(T))$. The computation is equivalent to solving a local problem $A_T \vecb{u}_T=\vecb{b}_{T}$, where
\begin{align}\label{eqn:def_AT}
  A_T := \left( \int_T \mathrm{div} (\vecb{\psi}_{j}) \mathrm{div} (\vecb{\psi}_{k})\dx\right)_{j,k}
  \quad \text{and} \quad
  b_T := \left( \int_T \mathrm{div}(\vecb{v}_h^{\mathrm{ct}}) \mathrm{div} (\vecb{\psi}_{j})\dx\right)_{j}.
\end{align}
Then $\vecb{u}_T=(u_{j})\in\mathbb{R}^{N_0}$ is indeed the vector of coefficients of the expansion of $\mathcal{R}\vecb{v}_h^{\mathrm{ct}}$ on $T$, i.e., $\mathcal{R}\vecb{v}_h^{\mathrm{ct}}|_T=\sum u_j \vecb{\psi}_j$. Moreover, one even does not need to solve local problems on each element. This computation can be achieved by affine transformation from the unit reference element (see Subsection \ref{sec:implementation} below).

\subsection{The $\vecb{P}_k \times P_0$-like scheme for arbitrary order}
\label{sec:reduced_pkp0}
For $k\geq d$, taking $\widehat{Q}_h={P}_0^{\mathrm{disc}}(\mathcal{T})\cap Q$ in Subsection~\ref{sec:reduced_scheme_general} results in a $\vecb{P}_k \times P_0$-like scheme.
For $k<d$, $\widetilde{P}_{k-1}^{\mathrm{disc}}(\mathcal{T})$ and $\widetilde{\vecb{RT}}_{k-1}^{\mathrm{int}}(\mathcal{T})$ play a similar role as $\widehat{Q}_h^\perp$ and $\vecb{V}_h^\mathrm{R}$ before in Subsection~\ref{sec:reduced_scheme_general}. Then the corresponding reduced scheme for $k<d$ seeks \(\widehat{\vecb{u}}_{h}=(\vecb{u}_h^{\mathrm{ct}}, \vecb{u}_h^{\mathrm{RT}_0})\in {\vecb{V}}_{h}^{\mathrm{ct}}\times(\vecb{RT}_{0}(\mathcal{T})\cap \vecb{H}_0(\mathrm{div},\Omega))\) such that
\begin{equation}\label{eq:galerkin_reducedform2}
\begin{aligned}
  \nu a_h(\widehat{\vecb{u}}_h-(\vecb{0},\mathcal{R}\vecb{u}_h^{\mathrm{ct}}), \widehat{\vecb{v}}_h-(\vecb{0},\mathcal{R}\vecb{v}_h^{\mathrm{ct}})) + b(\widehat{\vecb{v}}_h,\hat{p}_h)
  & = (\vecb{f}, \widehat{\vecb{v}}_h-(\vecb{0},\mathcal{R}\vecb{v}_h^{\mathrm{ct}})),\\
  b(\widehat{\vecb{u}}_h, q_h)
  & = 0,
\end{aligned}
\end{equation}
for all $\widehat{\vecb{v}}_h=(\vecb{v}_h^{\mathrm{ct}}, \vecb{v}_h^{\mathrm{RT}_0})\in {\vecb{V}}_{h}^{\mathrm{ct}}\times(\vecb{RT}_{0}(\mathcal{T})\cap \vecb{H}_0(\mathrm{div},\Omega)),q_h \in {P}_{0}^{\mathrm{disc}}(\mathcal{T})\cap Q$.
Note that \eqref{eq:galerkin_reducedform2} can be further reduced to a $\vecb{P}_{k} \times P_{0}$ system, after removing the $\vecb{RT}_{0}$ unknowns by static condensation since the $\vecb{RT}_0-\vecb{RT}_0$ block is diagonal \cite{li2021low}.

Here, $\vecb{V}_h^\mathrm{R}$ should be chosen as
\begin{align*}
\vecb{V}_h^\mathrm{R}:=\begin{cases}
\widetilde{\vecb{RT}}_{k-1}^{\mathrm{int}}(\mathcal{T}) \quad & k\geq d,\\
(\vecb{RT}_{0}(\mathcal{T})\cap \vecb{H}_0(\mathrm{div},\Omega))\oplus\widetilde{\vecb{RT}}_{k-1}^{\mathrm{int}}(\mathcal{T})\quad & k<d.
\end{cases}
\end{align*}
Indeed, this approach results in the same space as in Tables~\ref{tab:spaces2D} and \ref{tab:spaces3D} for $k\leq d$ but in a larger space for $k\geq d+1$. For $k=3,4$ in two dimensions, $\widetilde{\vecb{RT}}_{k-1}^{\mathrm{int}}(\mathcal{T})$ is exactly the sum of the chosen Raviart--Thomas subspaces from $2$ to $k$ in Table~\ref{tab:spaces2D}.

\begin{remark}
For a general order $k$, one can also use the non-divergence-free interior shape functions in
\cite[\S\,2.2.4]{Lehrenfeld2010} (two dimensions) and \cite[pp.\,106--107]{Zaglmayr2006} (three dimensions),
which forms a subspace of the Brezzi--Douglas--Marini spaces of order $k$, isomorphic to $\widetilde{\vecb{RT}}_{k-1}^\mathrm{int}(\mathcal{T})$, and makes the divergence operator
bijective onto $\widetilde{P}_{k-1}^\mathrm{disc}(\mathcal{T})$.
This subspace fulfills all the properties needed in our methods and plays a similar role as $\widetilde{\vecb{RT}}_{k-1}^\mathrm{int}(\mathcal{T})$.
\end{remark}

\subsection{Recovering $\tilde{p}_h=p_{h}-\hat{p}_{h}$ locally on each element $T$}
For simplicity, we use $k\geq d$ as an example.
Due to the inf-sup stability of $(\vecb{V}_h^{\mathrm{ct}}, \widehat{Q}_h)$, it holds that
$\pi_{\widehat{Q}_h}\mathrm{div}(\vecb{V}_h^{\mathrm{ct}})=\widehat{Q}_h$ and further $\mathrm{div}(\widehat{\vecb{V}}_{h})=\widehat{Q}_h$. Meanwhile, since $\widehat{\vecb{V}}_h$ is a subspace of $\vecb{V}_{h}$, the full pressure $p_{h}$ also satisfies
\begin{equation}\label{eq:fullpressure_with_reducedform}
\nu a_h(\vecb{u}_h, \vecb{v}_h) + b(\vecb{v}_h,{p}_h)
   = (\vecb{f}, \vecb{v}_h)\quad\text{ for all } \vecb{v}_h \in \widehat{\vecb{V}}_{h}.
\end{equation}
Subtracting \eqref{eq:fullpressure_with_reducedform} from \eqref{eq:galerkin_reducedform0} one gets, for all $\vecb{v}_h\in \widehat{\vecb{V}}_{h}$,
\begin{displaymath}
b(\vecb{v}_{h},p_{h})=b(\vecb{v}_{h},\hat{p}_{h}),
\end{displaymath}
which implies $\hat{p}_{h}=\pi_{\widehat{Q}_h}p_{h}$ and $\tilde{p}_h=p_{h}-\hat{p}_{h}\in \widehat{Q}_h^\perp$.

Now let us introduce the recover method for $p_{h}$ on $T\in\mathcal{T}$. Like Subsection~\ref{sec:reduced_scheme_general}, let $\{\vecb{\psi}_{j},j=1,\ldots,N_{0}\}$ be the set of basis functions of $\vecb{V}_h^\mathrm{R}$. Note also that $\mathrm{div}(\vecb{\psi}_{j}),j=1,\ldots,N_{0},$ are linearly independent and
\begin{displaymath}
\operatorname{span}\left\{\mathrm{div}(\vecb{\psi}_{j}),j=1,...,N_{0}\right\}=\widehat{Q}_h^\perp(T),
\end{displaymath}
by bijectivity of $\mathrm{div}|_{\vecb{V}_h^\mathrm{R}}$.
Hence the following equations form a local problem for $\tilde{p}_{h}$ on $T$:
\begin{displaymath}
(\tilde{p}_{h},\mathrm{div}(\vecb{\psi}_{j}))_{T}=-(\vecb{f},\vecb{\psi}_{j})_{T}+\nu a_{h}(\vecb{u}_{h},(\vecb{0},\vecb{\psi}_{j}))_{T},\quad j=1,\ldots,N_{0}.
\end{displaymath}
The above system forms the matrix $A_{T}$ also if one uses $\{\mathrm{div}\vecb{\psi}_{j}\}$ as the basis for representing $\tilde{p}_{h}$.

\subsection{Implementation}\label{sec:implementation}
This subsection comments on the algebraic structure of the full and the reduced scheme and some remarks on the efficient implementation.
Algebraically the full scheme solves a system of the form
\begin{align*}
\begin{pmatrix}
  A_{\mathrm{c}\mathrm{c}} & A_{\mathrm{R}\mathrm{c}}^{\top} & B_\mathrm{c}^{\top} \\
  -A_{\mathrm{R}\mathrm{c}} & A_{\mathrm{R}\mathrm{R}} & B_\mathrm{R}^{\top} \\
  B_\mathrm{c} & B_\mathrm{R} & 0 \\
\end{pmatrix}
\begin{pmatrix}
  U_\mathrm{c}\\
  U_\mathrm{R}\\
  P
\end{pmatrix}
=
\begin{pmatrix}
  F_\mathrm{c}\\
  F_\mathrm{R}\\
  0
\end{pmatrix},
\end{align*}
where \(U_\mathrm{c}\), \(U_\mathrm{R}\), and \(P\) are the coefficients for \(\vecb{V}_h^{\mathrm{ct}}\), \(\vecb{V}_h^\mathrm{R}\) and \(Q_h\), respectively.
The blocks \(A_{\mathrm{c}\mathrm{c}}\), \(B_{\mathrm{c}}\), \(F_\mathrm{c}\) are the standard blocks
related to $a$ and $b$ and $\vecb{f}$, respectively. The blocks \(B_{\mathrm{R}}\), \(F_\mathrm{R}\)
represent the $b$ and $\vecb{f}$ applied to functions from \(\vecb{V}_h^\mathrm{R}\). The stabilization block \(A_{\mathrm{R}\mathrm{R}}\)
refers to $a_h^D$ and is only needed when \(k<d\). In the lowest-order case $k=1$ it holds
\(A_{\mathrm{R}\mathrm{c}} = 0\), but for \(k > 1\) it corresponds to \((\Delta_\text{pw} \vecb{u}_h^{\mathrm{ct}}, \vecb{v}_h^{\mathrm{\mathrm{R}}})\).

For brevity, we restrict the remaining presentation to the case $k \geq d$, such that no $\vecb{RT}_0$ part is involved.
Given a representation matrix $R$ for the linear mapping $\mathcal{R} : \vecb{V}_h^{\mathrm{ct}} \rightarrow \vecb{V}_h^\mathrm{R}$ from the previous section, the reduced scheme solves instead
the much smaller system
\begin{align*}
  \begin{pmatrix}
    A_{\mathrm{c}\mathrm{c}} - A_{\mathrm{R}\mathrm{c}}^{\top} R + R^{\top} A_{\mathrm{R}\mathrm{c}}  & B_0^{\top} \\
    B_0 & 0 \\
  \end{pmatrix}
  \begin{pmatrix}
    U_\mathrm{c}\\
    P_0
  \end{pmatrix}
  =
  \begin{pmatrix}
    F_\mathrm{c} - R^\top F_\mathrm{R}\\
    0
  \end{pmatrix},
\end{align*}
where \(U_\mathrm{c}\) and \(P_0\) are the coefficients for \(\vecb{V}_h^{\mathrm{ct}}\) and \(P_0^\mathrm{disc}(\mathcal{T})\cap Q\), respectively.
The Raviart--Thomas part can be recovered by $U_\mathrm{R} = -R U_\mathrm{c}$.

Next, an efficient assembly of the representation matrix $R$ will be explained.
It is sparse and can be computed by solving local problems for each degree of freedom on each cell $T \in \mathcal{T}$.
In fact, for each basis function \(\vecb{v}_j^{\mathrm{ct}} \in \vecb{V}_h^{\mathrm{ct}}\) with support in \(T \in \mathcal{T}\), one has to solve a system formed by \eqref{eqn:def_AT},
and to add the solution to the \(j\)-th row of $R$.

These solutions can be obtained very efficiently
via the following strategy.
Assume $k=2$ in two dimensions as an example. We compute $\mathcal{R}$ via $\widetilde{\mathcal{R}}$ by the relationship ${\mathcal{R}}\vecb{v}_h^{\mathrm{ct}}=\widetilde{\mathcal{R}}(\mathrm{div}\vecb{v}_h^{\mathrm{ct}})$ for all $\vecb{v}_h^{\mathrm{ct}}\in \vecb{V}_h^{\mathrm{ct}}$. Note that $\mathrm{div}\vecb{V}_h^{\mathrm{ct}}|_{T}\subseteq P_1(T)=\mathrm{span}\left\{\varphi_1,\varphi_2,\varphi_3\right\}$
and a nodal interpolation allows to determine the coefficients $c_j$
in $\mathrm{div}\vecb{v}_h^{\mathrm{ct}}=\sum c_j\varphi_j$. Hence the problem reduces to
find the images of $\widetilde{\mathcal{R}}(\varphi_j|_T) = \sum_{j=1}^3 x_{T,j} \vecb{\psi}^{\mathrm{RT}_1}_{j}$ for $j=1,2,3$
and the particular basis functions \(\vecb{\psi}^{\mathrm{RT}_1}_{j}\) on $T$.
In principle one needs to solve the linear system of equations $A_T \vecb{x}_{T,j} = \vecb{b}_{T,j}$
with $A_T$ from \eqref{eqn:def_AT} and
\begin{align*}
  \vecb{b}_{T,j} := \left( \int_T \varphi_j \mathrm{div} (\vecb{\psi}^{\mathrm{RT}_1}_{k})\dx\right)_{k = 1,2,3}.
\end{align*}
However, employing the properties of the standard (Piola) transformations between the reference simplex $T^\text{ref}$
and the general simplex $T$, see ,e.g., \cite[Chapter II, \S 2]{BBF:book:2013} for details, it holds $A_T = J^{-1} A_{T^\text{ref}}$ and $\vecb{b}_{T,j}$ = $\vecb{b}_{T^\text{ref},j}$
where \(J\) is the determinant of the matrix of the affine transformation, i.e., \(J = d! \lvert T \rvert\).
In other words, it holds \(\vecb{x}_{T,j} = J \vecb{x}_{T^\text{ref},j}\) and the local systems for $j=1,2,3$ only have to be solved
once on the reference domain. Of course, the process can be generalized to larger $k$.

Similarly, for the pressure recovering process, one has $A_{T}^{-1}={J}A_{T^\mathrm{ref}}^{-1}$. Thus we only need to solve the inverse of $A_{T^\mathrm{ref}}$ and all inverses of $A_{T}$ can be obtained from it directly.
Then the process of pressure recovering only requires matrix-vector multiplications.

\section{Numerical results}
\label{sec:numerics}
This section confirms the theoretical findings
in two numerical examples. For $k < d$, the lowest-order Raviart--Thomas part is stabilized with $\alpha = 1$
in all numerical studies.

\input{example1_data.tex}
%% DATA FOR 3D EXAMPLE
%% generated with script ExamplePkRT_full.main(; dim = 3, order = order, alpha = alpha)

%% order = 1, alpha = 1

\pgfplotstableread[col sep=ampersand,row sep=\\]{
  ndof  &           L2u  &	  oL2u  &       L2p      &   oL2p   &       H1u      &    oH1u  &       L2uR     &    oL2uR &      L2divu    &  oL2divu  &       L2uD     &  oL2uD   \\
  144  &     1.911e-01  &    0.00  &     6.379e-02  &    0.00  &     1.715e+00  &    0.00  &     1.728e-01  &    0.00  &     3.905e-13  &    0.00  &     8.322e-17  &    0.00  \\
  729  &     1.433e-01  &    0.53  &     2.870e-02  &    1.48  &     1.712e+00  &    0.00  &     1.016e-01  &    0.98  &     1.918e-12  &    -2.94  &     2.456e-17  &    2.26  \\
 5210  &     4.539e-02  &    1.75  &     1.353e-02  &    1.15  &     8.770e-01  &    1.02  &     3.146e-02  &    1.79  &     8.536e-14  &    4.75  &     5.686e-18  &    2.23  \\
34113  &     1.210e-02  &    2.11  &     7.121e-03  &    1.03  &     4.553e-01  &    1.05  &     9.292e-03  &    1.95  &     4.691e-14  &    0.96  &     1.593e-18  &    2.03  \\
242743  &     2.920e-03  &    2.17  &     3.638e-03  &    1.03  &     2.265e-01  &    1.07  &     2.331e-03  &    2.11  &     3.414e-14  &    0.49  &     3.836e-19  &    2.18  \\
}\EXAMPLExTWOxORDERxONExFULL

%% order = 2, alpha = 1

\pgfplotstableread[col sep=ampersand,row sep=\\]{
  ndof  &           L2u  &	  oL2u  &       L2pr     &   oL2pr  &       L2p      &   oL2p   &       H1u      &    oH1u  &       L2uR     &    oL2uR &      L2divu    &  oL2divu  &       L2uD     &  oL2uD   \\
  299  &     1.883e-01  &    0.00  &     6.380e-02  &    0.00  &     1.165e-02  &    0.00  &     1.546e+00  &    0.00  &     1.762e-01  &    0.00  &     1.565e-16  &    0.00  &     2.573e-17  &    0.00  \\
  1455  &     3.927e-02  &    2.97  &     2.869e-02  &    1.51  &     3.039e-03  &    2.55  &     5.883e-01  &    1.83  &     3.737e-02  &    2.94  &     2.778e-16  &    -1.09  &     2.683e-18  &    4.29  \\
  9779  &     6.349e-03  &    2.87  &     1.353e-02  &    1.18  &     6.732e-04  &    2.37  &     1.833e-01  &    1.84  &     5.359e-03  &    3.06  &     1.055e-15  &    -2.10  &     3.778e-19  &    3.09  \\
 61669  &     8.321e-04  &    3.31  &     7.118e-03  &    1.05  &     1.774e-04  &    2.17  &     4.753e-02  &    2.20  &     6.685e-04  &    3.39  &     1.619e-15  &    -0.70  &     5.277e-20  &    3.21  \\
428790  &     1.048e-04  &    3.20  &     3.637e-03  &    1.04  &     4.562e-05  &    2.10  &     1.181e-02  &    2.15  &     8.581e-05  &    3.18  &     3.106e-15  &    -1.01  &     6.648e-21  &    3.20  \\
}\EXAMPLExTWOxORDERxTWOxREDUCED %

\pgfplotstableread[col sep=ampersand,row sep=\\]{
    ndof  &           L2u  &	  oL2u  &       L2p      &   oL2p   &       H1u      &    oH1u  &       L2uR     &    oL2uR &      L2divu    &  oL2divu  &       L2uD     &  oL2uD   \\
    453  &     1.883e-01  &    0.00  &     1.165e-02  &    0.00  &     1.546e+00  &    0.00  &     1.762e-01  &    0.00  &     2.130e-16  &    0.00  &     2.573e-17  &    0.00  \\
    2463  &     3.927e-02  &    2.78  &     3.039e-03  &    2.38  &     5.883e-01  &    1.71  &     3.737e-02  &    2.75  &     1.170e-15  &    -3.02  &     2.683e-18  &    4.01  \\
   18557  &     6.349e-03  &    2.71  &     6.732e-04  &    2.24  &     1.833e-01  &    1.73  &     5.359e-03  &    2.88  &     1.329e-15  &    -0.19  &     3.778e-19  &    2.91  \\
  124179  &     8.321e-04  &    3.21  &     1.774e-04  &    2.11  &     4.753e-02  &    2.13  &     6.685e-04  &    3.29  &     1.859e-15  &    -0.53  &     5.277e-20  &    3.11  \\
  893527  &     1.048e-04  &    3.15  &     4.562e-05  &    2.06  &     1.181e-02  &    2.12  &     8.581e-05  &    3.12  &     1.196e-14  &    -2.83  &     6.648e-21  &    3.15  \\
}\EXAMPLExTWOxORDERxTWOxFULL

%% order = 3, alpha = 0

\pgfplotstableread[col sep=ampersand,row sep=\\]{
  ndof  &           L2u  &	  oL2u  &       L2pr     &   oL2pr  &       L2p      &   oL2p   &       H1u      &    oH1u  &       L2uR     &    oL2uR &      L2divu    &  oL2divu  &       L2uD     &  oL2uD   \\
  622  &     7.212e-02  &    0.00  &     6.378e-02  &    0.00  &     1.698e-03  &    0.00  &     8.310e-01  &    0.00  &     3.280e-02  &    0.00  &     3.377e-16  &    0.00  &     5.540e-18  &    0.00  \\
  3189  &     1.049e-02  &    3.54  &     2.869e-02  &    1.47  &     2.440e-04  &    3.56  &     2.027e-01  &    2.59  &     3.999e-03  &    3.86  &     6.914e-16  &    -1.32  &     4.456e-19  &    4.63  \\
 22488  &     5.961e-04  &    4.40  &     1.353e-02  &    1.15  &     2.894e-05  &    3.27  &     2.596e-02  &    3.16  &     3.099e-04  &    3.93  &     1.826e-15  &    -1.49  &     2.487e-20  &    4.43  \\
145181  &     3.717e-05  &    4.46  &     7.118e-03  &    1.03  &     3.850e-06  &    3.24  &     3.364e-03  &    3.29  &     2.294e-05  &    4.19  &     3.531e-15  &    -1.06  &     1.944e-21  &    4.10  \\
}\EXAMPLExTWOxORDERxTHREExREDUCED %

\pgfplotstableread[col sep=ampersand,row sep=\\]{
    ndof  &           L2u  &	  oL2u  &       L2p      &   oL2p   &       H1u      &    oH1u  &       L2uR     &    oL2uR &      L2divu    &  oL2divu  &       L2uD     &  oL2uD   \\
    1018  &     7.212e-02  &    0.00  &     1.698e-03   &   0.00  &     8.310e-01   &   0.00  &     3.280e-02   &   0.00  &     2.875e-10  &    0.00  &     5.540e-18   &   0.00  \\
    5781  &     1.049e-02  &    3.33  &     2.440e-04   &   3.35  &     2.027e-01   &   2.44  &     3.999e-03   &   3.64  &     1.833e-10  &    0.78  &     4.456e-19   &   4.35  \\
   45060  &     5.961e-04  &    4.19  &     2.894e-05   &   3.11  &     2.596e-02   &   3.00  &     3.099e-04   &   3.74  &     2.373e-10  &    -0.38  &     2.487e-20  &    4.22  \\
  305921  &     3.717e-05  &    4.35  &     3.850e-06   &   3.16  &     3.364e-03   &   3.20  &     2.294e-05   &   4.08  &     1.947e-11  &    3.92  &     1.944e-21   &   3.99  \\
}\EXAMPLExTWOxORDERxTHREExFULL

\begin{figure}
 \begin{tikzpicture}
  \begin{axis}[
    legend style={font=\scriptsize, at={(0,0)}, anchor=south west},
    xmode=log,
    ymode=log,
    every axis plot/.append style={thick},
    width=7cm,
    height=6cm,
    title = $\| \vecb{u} - (\vecb{u}_h^\mathrm{ct}+\vecb{u}_h^\mathrm{R})\|$,
    xlabel={ndof},
    ylabel={}]
    \addplot[color = blue, mark = pentagon*] table[x=ndof,y=L2u] {\EXAMPLExTWOxORDERxONExFULL};
    \addlegendentry{$k = 1$}
    \addplot[color = red, mark = square*] table[x=ndof,y=L2u] {\EXAMPLExTWOxORDERxTWOxFULL};
    \addlegendentry{$k = 2$}
    \addplot[color = olive, mark = *] table[x=ndof,y=L2u] {\EXAMPLExTWOxORDERxTHREExFULL};
    \addlegendentry{$k = 3$}

    \addplot[color = red, mark = square*, dashed] table[x=ndof,y=L2u] {\EXAMPLExTWOxORDERxTWOxREDUCED};
    \addplot[color = olive, mark = *, dashed] table[x=ndof,y=L2u] {\EXAMPLExTWOxORDERxTHREExREDUCED};

    \addplot[domain=1e2:1e6, color=gray, dotted]{1e1/pow(x,2/3)};
    \addplot[domain=1e2:1e6, color=gray, dotted]{4e1/pow(x,3/3)};
    \addplot[domain=1e2:1e6, color=gray, dotted]{2e2/pow(x,4/3)};
  \end{axis}
  \end{tikzpicture}
  \begin{tikzpicture}
   \begin{axis}[
     legend style={font=\scriptsize, at={(0,0)}, anchor=south west},
     xmode=log,
     ymode=log,
     every axis plot/.append style={thick},
     width=7cm,
     height=6cm,
     title = $\| \nabla(\vecb{u} - \vecb{u}^1_h) \|$,
     xlabel={ndof},
     ylabel={}]
     \addplot[color = blue, mark = pentagon*] table[x=ndof,y=H1u] {\EXAMPLExTWOxORDERxONExFULL};
     \addlegendentry{$k = 1$}
     \addplot[color = red, mark = square*] table[x=ndof,y=H1u] {\EXAMPLExTWOxORDERxTWOxFULL};
     \addlegendentry{$k = 2$}
     \addplot[color = olive, mark = *] table[x=ndof,y=H1u] {\EXAMPLExTWOxORDERxTHREExFULL};
     \addlegendentry{$k = 3$}

     \addplot[color = red, mark = square*, dashed] table[x=ndof,y=H1u] {\EXAMPLExTWOxORDERxTWOxREDUCED};
     \addplot[color = olive, mark = *, dashed] table[x=ndof,y=H1u] {\EXAMPLExTWOxORDERxTHREExREDUCED};

    \addplot[domain=1e2:1e6, color=gray, dotted]{1e1/pow(x,1/3)};
    \addplot[domain=1e2:1e6, color=gray, dotted]{6e1/pow(x,2/3)};
    \addplot[domain=1e2:1e6, color=gray, dotted]{4e2/pow(x,3/3)};
   \end{axis}
   \end{tikzpicture}

   \begin{tikzpicture}
    \begin{axis}[
      legend style={font=\scriptsize, at={(0,0)}, anchor=south west},
      xmode=log,
      ymode=log,
      every axis plot/.append style={thick},
      width=7cm,
      height=6cm,
      title = $\| \vecb{u}_h^R \|$,
      xlabel={ndof},
      ylabel={}]
      \addplot[color = blue, mark = pentagon*] table[x=ndof,y=L2uR] {\EXAMPLExTWOxORDERxONExFULL};
      \addlegendentry{$k = 1$}
      \addplot[color = red, mark = square*] table[x=ndof,y=L2uR] {\EXAMPLExTWOxORDERxTWOxFULL};
      \addlegendentry{$k = 2$}
      \addplot[color = olive, mark = *] table[x=ndof,y=L2uR] {\EXAMPLExTWOxORDERxTHREExFULL};
      \addlegendentry{$k = 3$}

      \addplot[color = red, mark = square*, dashed] table[x=ndof,y=L2uR] {\EXAMPLExTWOxORDERxTWOxREDUCED};
      \addplot[color = olive, mark = *, dashed] table[x=ndof,y=L2uR] {\EXAMPLExTWOxORDERxTHREExREDUCED};

      \addplot[domain=1e2:1e6, color=gray, dotted]{7e0/pow(x,2/3)};
      \addplot[domain=1e2:1e6, color=gray, dotted]{3e1/pow(x,3/3)};
      \addplot[domain=1e2:1e6, color=gray, dotted]{1e2/pow(x,4/3)};
    \end{axis}
    \end{tikzpicture}
   \begin{tikzpicture}
    \begin{axis}[
      legend style={font=\scriptsize, at={(0,0)}, anchor=south west},
      xmode=log,
      ymode=log,
      every axis plot/.append style={thick},
      width=7cm,
      height=6cm,
      title = $\| p - p_h\|$,
      xlabel={ndof},
      ylabel={}]
      \addplot[color = blue, mark = pentagon*] table[x=ndof,y=L2p] {\EXAMPLExTWOxORDERxONExFULL};
      \addlegendentry{$k = 1$}
      \addplot[color = red, mark = square*] table[x=ndof,y=L2p] {\EXAMPLExTWOxORDERxTWOxFULL};
      \addlegendentry{$k = 2$}
      \addplot[color = olive, mark = *] table[x=ndof,y=L2p] {\EXAMPLExTWOxORDERxTHREExFULL};
      \addlegendentry{$k = 3$}

      \addplot[color = red, mark = square*, dashed] table[x=ndof,y=L2p] {\EXAMPLExTWOxORDERxTWOxREDUCED};
      \addplot[color = olive, mark = *, dashed] table[x=ndof,y=L2p] {\EXAMPLExTWOxORDERxTHREExREDUCED};

    \addplot[domain=1e2:1e6, color=gray, dotted]{1e-1/pow(x,1/3)};
    \addplot[domain=1e2:1e6, color=gray, dotted]{2e-1/pow(x,2/3)};
    \addplot[domain=1e2:1e6, color=gray, dotted]{6e-1/pow(x,3/3)};
    \end{axis}
    \end{tikzpicture}
    \caption{\label{fig:convhistory_example2}Example 2: Convergence history for $\| \vecb{u} - (\vecb{u}_h^\mathrm{ct}+\vecb{u}_h^\mathrm{R})\|$ (top left), $\| \nabla(\vecb{u} - \vecb{u}^1_h)\|$ (top right), $\| \vecb{u}^R \|$ (bottom left), and $\| p - p_h \|$ (bottom right) for different $k$.
    The solid lines coresspond to the full schemes, while dashed lines coresspond to the reduced schemes. The slopes of the gray dotted lines coresspond to the expected optimal order of the curve(s) right above them.}
  \end{figure}

\subsection{Example 1 - Two-dimensional planar lattice flow}
Consider the planar lattice flow
\begin{align*}
    \vecb{u} = \begin{pmatrix}
   	 		\sin(2\pi x) \sin(2 \pi y) \\
    		\cos(2\pi x) \cos(2 \pi y)
    	\end{pmatrix}
    	\quad \text{and} \quad
    p = (\cos(4 \pi x) - \cos(4 \pi y)) / 4
\end{align*}
with \(\vecb{f}\) chosen such that \((\vecb{u}, p)\) solves the Stokes problem with \(\nu = 10^{-3}\).

Figure~\ref{fig:convhistory_example1} displays the convergence histories for the velocity
and pressure errors and confirms that all methods converge optimally for order $k=1,2,3,4$
according to Theorems~\ref{thm:velocity_estimate} and \ref{thm:pressure_estimate}.
Further studies, which are not presented for the sake of brevity, showed also that
$\vecb{u}_h^{\mathrm{ct}} + \vecb{u}_h^{\mathrm{R}}$ is exactly divergence-free
and that the scheme is pressure-robust, i.e., the velocity does not change for other values of $\nu$.
The reduced schemes for $k=2,3,4$ also produce optimally converging velocity
and pressure approximations. Here the reduced scheme refers to the $\vecb{P}_k-P_0$-like scheme in Subsection~\ref{sec:reduced_pkp0}, while the full scheme refers to \eqref{eq:fullscheme} with the enrichment space discussed in Subsection~\ref{sec:enrichment_space_choices}. For $k=2$ the solutions of the full scheme and the reduced scheme
coincide, while for $k>2$ the solutions slightly differ in the sense that
the $\vecb{V}_h^{\mathrm{R}}$ is a bit larger in case of
the reduced scheme.

\subsection{Example 2 - A three-dimensional example with analytic solutions} On $\Omega=(0,1)^3$, the solution is prescribed as
\begin{displaymath}
\vecb{u}=\frac{1}{2\pi}\operatorname{curl}\left\{[\sin(\pi x)\sin(\pi y)]^{2}\sin(\pi z) \vecb{e}_{3}\right\} \quad \text{and} \quad p=\sin(x)\sin(y)\sin(z)-(1-\cos1)^3,
\end{displaymath}
with $\vecb{e}_{3}=(0,0,1)^{\top}$. Again, we consider a Stokes equation with $\nu=10^{-3}$ by choosing a suitable body force.

Figure~\ref{fig:convhistory_example2} displays the convergence histories for the velocity
and pressure errors, where the enrichment space discussed in Subsection~\ref{sec:enrichment_space_choices} is employed. The expected convergence orders from Theorems~\ref{thm:velocity_estimate}
and \ref{thm:pressure_estimate} are obtained. The reduced scheme for $k=2$ and $k=3$ produces
exactly the same solution as the full scheme, while solving a smaller
linear system of equations.

\section{Summary and outlook}

This paper presents a novel way how to stabilize the Scott--Vogelius finite element
method for arbitrary polynomial degree $k$ and general shape-regular simplicial meshes that preserves
optimal convergence and the divergence-free property of the velocity.
This is realized by enriching the velocity space with carefully chosen Raviart--Thomas
functions that stabilize the orthogonal complement of a small enough
sub-pressure space that is known to be inf-sup stable.
Finally, a reduced scheme is studied which only involves the degrees of freedom of the $\vecb{H}^1$-conforming velocity and a piecewise constant pressure.

In the future, extensions to Navier--Stokes problems will be studied, which is able to preserve some conservation properties (e.g., conservation of linear momentum) and is $Re$-semi-robust. The latter property refers to estimates where the constants do not
blow up as the Reynolds number becomes large.

\bibliographystyle{siam}
\bibliography{lit}

\end{document}